\documentclass[leqno]{siamart1116}
\usepackage{amssymb,amsfonts,latexsym,epsfig,euscript,amsmath,bbm,tikz}
\usepackage{color}
 \usepackage[latin1]{inputenc}
\numberwithin{theorem}{section}
\numberwithin{equation}{section}

\newcommand{\diag}{\mathop{\mathrm{diag}}}

\def\dist{\textit{\textsf{d}}} 

\def\uno{\mathbbm 1}
\def\F{\Omega}

\def\overrightarrow#1{\overset{\rightarrow}{#1}}
\def\overleftarrow#1{\overset{\leftarrow}{#1}}

\renewcommand{\tilde}{\widetilde}

\def\note#1{ {\sffamily \textcolor{blue}{#1}} }

\newenvironment{rev}{\color{blue}}{\color{black}}

\usepackage[normalem]{ulem} 

\begin{document}
\date{}

\title{On the stability of network indices defined by means of matrix functions\thanks{Author's accepted version: this is the peer-reviewed version of this manuscript, which is now   published on SIAM Journal on Matrix Analysis and Applications \url{https://doi.org/10.1137/17M1133920}. \funding{The work of S.P. has been supported by Charles University Research program No. UNCE/SCI/023, and by INdAM, GNCS (Gruppo Nazionale per il Calcolo Scientifico). The work of F.T. was funded by the European Union's Horizon 2020 research and innovation programme under the MarieSk\l odowska-Curie individual fellowship ``MAGNET'' grant agreement no.\ 744014.}}}

\author{Stefano Pozza\thanks{Department of Numerical Mathematics,
Faculty of Mathematics and Physics,
Charles University in Prague (\email{pozza@karlin.mff.cuni.cz}).}. 
\and Francesco Tudisco \thanks{Department of Mathematics and Statistics, 
 University of Strathclyde, Glasgow, G11XH, UK (\email{tudisco.francesco@gmail.com})}}

\maketitle

\begin{abstract}
Identifying important components in a network  is one of the major goals of network analysis. 
Popular and effective measures of importance of a node or a set of nodes are defined in terms of suitable entries of functions of matrices $f(A)$.  These kinds of measures 
are particularly relevant as they are able to capture the global structure of connections involving a node. However, computing the entries of $f(A)$ requires a significant computational effort. 
In  this work we address the problem of estimating the changes in the entries of $f(A)$ with respect to changes in the edge structure. Intuition suggests that, if the topology or the overall weight of the  connections in the new graph $\tilde G$ 
are  not significantly distorted, relevant components in $G$ maintain their leading role in $\tilde G$.  We propose several bounds giving mathematical reasoning to such intuition and showing, in particular,  that the magnitude of the variation of the entry $f(A)_{k\ell}$ decays exponentially with the shortest-path distance in $G$ that separates either $k$ or $\ell$ from the set of nodes touched by the edges that are perturbed. Moreover, we propose a simple method that exploits the computation of $f(A)$ to simultaneously compute the all-pairs shortest-path distances of $G$, with essentially no additional cost.
The proposed bounds are particularly relevant when the nodes whose edge connection tends to change more often or tends to be more often affected by noise have marginal role in the graph and are distant from the most central nodes. 
\end{abstract}



\section{Introduction}\label{sec.intro} 
Networks and datasets of large dimension arise naturally in a number of diversified applications, ranging from biology and  chemistry to computer science, physics and engineering, see, e.g., \cite{alon2006introduction,boccaletti2006complex,  brin1998anatomy, crofts2009weighted, grindrod2004review, morrison2006lock,  onnela2007structure}.  Being able to recognize important components within a vast amount of data is one of the main goals of the analysis of networks. As a network can be uniquely  identified with an adjacency matrix, 
many efficient mathematical  and numerical strategies for revealing relevant components employ tools from numerical linear algebra and matrix analysis. Important examples include locations of clusters of data points \cite{estrada2017core, kloster2016graph, mercado2016clustering, ng2001spectral}, detection of communities \cite{fasino2014algebraic,fasino2016generalized, newman2006finding} and ranking of nodes and edges \cite{benzi2013ranking, estrada2010network, langville2012s}. 

To address  the latter range of problems, a popular approach is to employ  the concepts known as centrality and communicability of the nodes of a network. These two attributes describe a certain measure of importance of nodes and edges in a network. Many commonly used and successful models for communicability and centrality measures are based on matrix eigenvectors. These models quantify the importance of a node in terms of the importances  of its neighbors, thus relying on the local behavior around the node. In this work we focus on another common class of models for centrality and communicability measures based, instead, on matrix functions \cite{bonacich1987power,freeman1978centrality}. This latter class of models is particularly informative and effective as, unlike the eigenvector-based models previously mentioned, the use of matrix functions allows to capture the global structure of connections involving a node.
However the matrix function approach requires a significantly larger computational cost. This is particularly prohibitive, for example, when the network changes and the importance of nodes or edges has  to be updated or when the network data is affected by noise and the importances can be biased.
When the perturbation on the network yields a low-rank update of the matrix, techniques for an efficient update of  the matrix function have been recently developed in \cite{beckermann2017low}. However, in general, each change in the network requires a complete re-computation of the matrix function to obtain the updated measure. 
On the other hand, in many applications one needs to know only ``who are'' the first few most important nodes in the graph 
and how stable they are with respect to noise or edge perturbations. 
Moreover, the nodes whose edge connection tends to change more often or is more likely to be affected by noise are those having a marginal role in the graph \cite{langville2012s}. 

Intuition suggests that, if the topology  or the overall weight of the connections in the new ``perturbed'' graph $\tilde G$ are  not significantly distorted, relevant 
nodes in the original graph $G$ maintain their leading role in $\tilde G$. 
In this paper we provide mathematical support for this intuition by analyzing the stability of  network measures based on matrix functions with respect to edge changes. 
By exploiting the theory of Faber polynomials and the recent literature on functions of banded matrices \cite{ benzi_boito_2014,benzi_razouk, pozza.simoncini.2016,suetin}, we propose a number of bounds showing that the magnitude of the variation of the centrality of node $k$ or the communicability between nodes $k$ and $\ell$ decays exponentially with the distance in the graph that separates either $k$ or $\ell$  from the set of nodes touched by the edges that are perturbed. This implies, for example, that if changes in the edge structure occur in a relatively small and peripheral network area -- in the sense that the perturbed edges involve only nodes being far from the most relevant ones -- then the set of leading nodes remains unchanged. 

The use of matrix functions for the analysis of networks has been originally proposed for undirected graphs \cite{estrada2005subgraph}. The extension to directed networks has been then proposed following different avenues: In \cite{fenu:2013}, for example, functions of the non-symmetric adjacency matrix are considered, whereas  in \cite{benzi2013ranking} functions of modified adjacency matrices are proposed in order to take into account the hub and authority nature of nodes in directed networks. In particular, this second formulation involves functions of symmetric block matrices of the form $\begin{pmatrix} 0& A^T \\ A & 0  \end{pmatrix}$. Although the techniques that we develop here can be transferred to this setting, this analysis goes beyond the scope of this paper and will be the subject of  future work.  
Here we focus on the case of functions of adjacency matrices for the general case of directed networks (i.e. possibly non-symmetric matrices) and discuss the undirected setting as a particular case. 

We organize the discussion as follows: The next section reviews some central   concepts and properties  we shall use throughout the present work, in particular the notions of $f$-centrality and $f$-communicability.
Section \ref{sec:motivations} is devoted to give detail about  our motivating ideas. Then, in  Section \ref{sec:faber}, we review the relevant theory about Faber polynomials. In Section \ref{sec:main:results} we state and prove our main results where we provide a number of bounds on the absolute variation of the centrality and the communicability measures of nodes $k$ and $\ell$ based on the matrix function $f(A)$ when some edges are modified in $G$. The bounds are given  in terms of the distances in $G$ that separate $k$ and $\ell$ from the set of perturbed edges and for two important network matrices: the adjacency matrix and the normalized (random walk) adjacency matrix. 
We consider both directed and undirected weighted networks and we give particular attention to the case of the exponential and the resolvent function, as they often arise in the related literature on complex networks. We also provide a simple algorithm that exploits the computation of the entries of $f(A)$ to simultaneously address the all-pairs shortest-path distances of the graph, at essentially no additional cost.
Finally, in Section \ref{sec:experiments}, we provide several numerical experiments where the behavior of the proposed computational strategy as well as the accuracy of the proposed  bounds is tested on some example networks, both synthetically generated and borrowed from real-world applications.

\section{Network properties and matrix functions}\label{sec:network-things}
One of the major goals of data analysis is to identify important 
nodes in a network  $G=(V,E)$ by exploiting the topological structure of connections between nodes. In order to address this matter from the mathematical point of view one needs a quantitative definition of the importance of a node $k$ or a pair of nodes $(k,\ell)$, thus leading to concepts such as the nodes centrality and the nodes communicability. Although these quantities have a long history, dating back to the early 1950s, recent years have seen the introduction of many new centrality scores based on the entries of certain function of matrices \cite{benzi2013ranking,estrada2010network,estrada2005subgraph,grindrod2014dynamical}. The idea behind such metrics is to measure the relevance of a node, for example, by quantifying  the number of subgraphs of $G$ that involve that node. In order to better perceive these concepts, we first introduce some preliminary graph notation.

Let $G=(V,E,\omega)$ be a weighted and directed graph where 
$V=\{1,\dots,N\}$ is the finite set of nodes, $E\subseteq V\times V$ is the set of edges and $\omega:E\to\mathbb R_+$ is a positive weight function.
In order to allow broader generality, in our notation undirected graphs will be particular cases of directed graphs, with the special requirement that an edge from $k$ to $\ell$ exists if and only if the edge from $\ell$ to $k$ does as well. With this convention, in what follows,  properties and results holding for a graph $G$ are implicitly understood for both directed and undirected graphs, unless explicitly specified otherwise.

To any graph  $G=(V,E,\omega)$ corresponds an entry-wise nonnegative adjacency matrix
$A=(a_{ij})\in {\mathbb R}_+^{N\times N}$ defined by
$$
a_{ij} = \begin{cases}
          \omega(e)&\text{if $i,j$ are starting and ending points of $e\in E$, respectively}\\
          0 & \text{otherwise}
         \end{cases} \, .
$$
Vice-versa, to any nonnegative $N\times N$ matrix $A=(a_{ij})\geq 0$ corresponds a graph $G=(V,E,\omega)$ such that $V = \{1, \dots, N\}$, $E$ is the set of pairs $ij$ such that $a_{ij}>0$ and $\omega(ij)=a_{ij}$ for any $ij\in E$.  
Clearly, the graph is undirected if and only if the associated adjacency matrix $A$ is symmetric.  

A self-loop in $G$ is an edge that goes from a node to itself. Given two nodes $k,\ell\in V$, a walk in $G$ from $k$ to $\ell$ is an ordered sequence of edges $\mathcal W=\mathcal W(k,\ell)=\{e_1, \dots, e_r\}\subseteq E$ such that $k$ is the starting point of $e_1$, $\ell$ is the endpoint of $e_r$ and, for any $i=1,\dots, r-1$, the endpoint of $e_i$ is the starting point of $e_{i+1}$. The length of a walk  is the number of edges forming the sequence (repetitions are allowed) and is denoted by $|\mathcal W|$. 
The length of the shortest walk from $k$ to $\ell$ is called the (geodesic or shortest-path) distance in $G$ from $k$ to $\ell$ and is denoted hereafter by $\dist_G(k,\ell)$. 
If there is no walk in $G$ connecting the pair $(k,\ell)$, we set $\dist_G(k,\ell)=+\infty$.
The diameter of $G$ is the longest shortest-path distance between any two nodes.  Given a set $S\subseteq V$ and a node $k\in V$, 
we let 
$$\dist_G(k,S) = \min_{s \in S}\dist_G(k,s)\quad \text{and} \quad \dist_G(S,k) = \min_{s \in S}\dist_G(s,k)\, ,$$
with the convention that $\dist_G(k,k)=0$ and thus $\dist_G(k,S)=\dist_G(S,k)=0$, for any $k\in S$.

A graph is said to be \textit{strongly connected} if $\dist_G(k,\ell)$ is a finite number, for any two 
nodes $k$ and $\ell$. 
We remark that in the graph theoretic literature a strongly connected undirected graph is often referred to as \textit{connected}. However, in line with our choice of notation, we use the term strongly connected for both directed and undirected graphs.  

The weight of a walk $\mathcal W$ is defined by
$$\omega(\mathcal W) = \prod_{e\in \mathcal W}\omega(e)\, .$$
This quantity has a natural matrix representation.  In fact, if $A$ is the adjacency matrix of $G$, then for any walk $\mathcal W=\mathcal W(k,\ell)$, there exists a sequence of nodes $u_1,\dots, u_n$, such that
\begin{equation}\label{eq:walk_weight} 
 \omega(\mathcal W) = a_{ku_1}a_{u_1u_2}\cdots a_{u_{n-1}u_{n}}a_{u_{n}\ell}\, .
\end{equation}

The preceding formula shows that the powers of the  adjacency matrix $A$ can be used to count the ``weighted number'' of walks of different lengths in $G$. 
More precisely, if $n$ is a positive integer and $\Omega_n(k,\ell) = \{\mathcal W(k,\ell):|\mathcal W(k,\ell)|=n\}$ is the set of walks from $k$ to $\ell$ of length exactly $n$, then one can easily observe that  
$(A^n)_{k \ell} = \sum_{\mathcal W \in \Omega_n(k,\ell)}\omega(\mathcal W)$.  
It is worth noting that, regardless of the edge weight function $\omega:E\to\mathbb R_+$, such characterization of the entries of $A^n$ implies that
\begin{equation}\label{eq:matrix:power}
    (A^n)_{k \ell} = 0, \quad \text{ for every } n < \dist_G(k,\ell)\, .
\end{equation}
This property will be one of the key tools of our forthcoming analysis.

A matrix function can be defined in a number of different but equivalent ways (see,  e.g.,  \cite{higham2008functions}).  Here we adopt the power series representation as it has a direct interpretation in terms of network properties: 
Consider a function $f:\mathbb C\to\mathbb C$ that admits the power representation $f(z) =\sum_{n\geq 0}\theta_n\, z^n$ for any $z\in \mathbb C$ such that $|z|\leq r$, with $r>0$. 
If the convergence radius $r$ is larger than the spectral radius of $A$, then we define the matrix function $f(A)$ as $f(A)=\sum_{n\geq 0}\theta_n\, A^n$; see,  e.g., \cite[Theorem 4.7]{higham2008functions}.
Given such a function $f$, the importance of a node in a network can be quantified in terms of certain entries of the matrix $f(A)$. This idea was firstly introduced by Estrada and Rodriguez-Vasquez in \cite{estrada2005subgraph}, for the particular choice $f(z)=\exp(z)$,  and then developed and extended in many subsequent works, see, e.g., \cite{aprahamian2016matching,benzi2013ranking,estrada2010network,fenu:2013} and the references therein. We thus adopt the following definition: 
\begin{definition}
Let $A\in \mathbb R_+^{N\times N}$ be the adjacency matrix of a graph and let $\rho(A)$ be its spectral radius. Let  $f:\mathbb C\to\mathbb C$ be such that  $f(z)=\sum_{n\geq 0}\theta_n\, z^n$, for any $|z|\leq r$ with $r>\rho(A)$. Further assume $\theta_n > 0$ for all $n$. 
The $f$-centrality of the node $k\in V$ is the quantity $f(A)_{kk}$. The $f$-communicability from node $k$  to node $\ell$ is the quantity $f(A)_{k\ell}$.  
\end{definition}

The centrality of a node is a measure of its importance as a component in the graph. Using \eqref{eq:walk_weight} one easily realizes that the quantity $f(A)_{kk} =\sum_{n\geq 0} \theta_n (A^n)_{kk}$ is the sum of the weights of all the possible closed walks from $k$ to itself, scaled  by the positive coefficients $\theta_0,\theta_1,\theta_2,\dots$. If $f(A)_{kk}$ is large, then many closed walks
with large weight
 pass by the  node $k\in V$ and thus $k$ is an important node in $G$. 

The communicability of a pair of nodes is a measure of the robustness of the communication between the pair. Arguing as before, if $f(A)_{k\ell}$ is large then many walks with large weight start in $k$ and end up in $\ell$. Thus the connection between these two nodes is likely to be not affected by possible breakdowns in the edge structure of the network, that is any message sent from $k$ towards $\ell$ is very likely to reach its destination. 

As mentioned before, typical choices of the function $f$ in the context of networks analysis are the exponential and the resolvent functions \cite{estrada2010network,estrada2005subgraph}, respectively obtained by choosing  $\theta_n = 1/n!$ or $\theta_n=\alpha^n$, where $0<\alpha<\rho(A)$ and  $\rho(A)$ is the spectral radius of $A$. Namely,
\begin{equation}\label{eq:exp_and_res}
 \exp(A) = \sum_{n\geq 0} \frac 1 {n!}\, A^n\quad \text{and} \quad r_\alpha(A) = \sum_{n\geq 0} \alpha^n \, A^n =(I-\alpha A)^{-1}\, .
\end{equation}

\section{Motivations}\label{sec:motivations}
This work is concerned with  the problem of estimating the changes in the entries of $f(A)$ with respect to ``small'' changes in the entries of $A$. However, for our purposes, the concept of being small is not necessarily  related with the norm nor the spectrum of the perturbation, rather we assume that a small number of entries are modified in $A$ via a sparse matrix $\delta A$. This form of perturbation has the following network interpretation: if $A$ is a square nonnegative matrix of order $N$, 
then $A$ is the adjacency matrix of a uniquely defined graph $G=(V,E,\omega)$ and 
adding the sparse noise $\delta A\in \mathbb R^{N\times N}$ to $A$ is equivalent to adding, removing or modifying the weight of the edges in a set $\delta E\subseteq V\times V$, with $|\delta E|\ll |E|$. We obtain in this way a new graph $\tilde G = (V, \tilde E, \tilde \omega)$ 
with adjacency matrix $\tilde A = A+\delta A$, such that $\tilde E \subseteq E\cup \delta E$ and  $\tilde \omega:\tilde E \to \mathbb R_+$ coincides with $\omega$ on $\tilde E\setminus \delta E$.  Although the norm of $\delta A$ can be large, intuition suggests that, if the topology or the weight of the connections in the new graph $\tilde G$ are  not significantly distorted, relevant  nodes and edges in $G$ maintain their leading role in $\tilde G$.  Providing mathematical evidence in support of this intuition is one of the  main objectives of the present work, where we show that the magnitude of the variation of the entry $f(A)_{k\ell}$ decays exponentially with the distance in $G$ that separates either $k$ or $\ell$ from the set of nodes touched by the new edges $\delta E$. 

This is of particular interest when addressing measures of $f$-centrality or  $f$-communicability for  large networks. To fix ideas, let us focus on the centrality  case and consider the example case where the network represents a data set where interactions evolve in time. Let $A$ be the adjacency matrix of the current graph $G$ and $\tilde A = A+\delta A$ be the adjacency of the graph $\tilde G$ in the next time stamp.  Computing the diagonal entries of $f(A)$ is a costly operation and, in the general case, knowing the importance of nodes in $\tilde G$ requires computing the entries of $f(\tilde A)$ almost from scratch.  However,  often one needs to know only ``who are'' the first few most important nodes in the graph, whereas the nodes whose edge connection tends to change more often are those having a marginal role (i.e. those with small centrality score) in the graph \cite{langville2012s}, and typically we expect that the distance in $G$ from important nodes and nodes having a marginal role is large.

To  gain further intuition,  in what follows we consider an example model where an edge exists from node $i$ to node $j$ with a probability being exponentially dependent on the difference between  the importances of $i$ and $j$.  This is a form  of  ``logistic preferential attachment'' model,  where edge distribution follows an exponential rather than a more common power law. The reasons for this choice are purely expository, as the logistic function simplifies the computations we discuss below. 
Let $c:V\to \mathbb [0,1]$ be a centrality function measuring the importances of nodes. 
Assume that an edge $e$ from $i$ to $j$ with weight $\omega(e)=1$ exists with probability 
\begin{equation}\label{eq:random_model}
    \mathbb P(i\to j) =  s_\alpha(c_j-c_i), \quad \text{where} \quad  s_\alpha(x)=\frac{1}{1+e^{-\alpha\, x}} \, .
\end{equation}
The parameter $0<\alpha\leq 1$ can be used to vary the slope of the sigmoid function $s_\alpha$ and thus to  tune the growth rate of $s_\alpha$ towards $1$ or $0$, as $x$ increases or decreases, respectively. 

This model is assuming that highly ranked nodes are very unlikely pointing to nodes with low rank, whereas the reverse implication is very likely to occur.  This kind of phenomenon is relatively common in real-world networks \cite{barabasi1999emergence} and it is at the basis of several ranking models \cite{langville2012s}.  
Let $\binom{a}{b} = \frac{a!}{(a-b)!b!}$ be the binomial coefficient, we have 
\begin{theorem}\label{pro:prob_bound}
Let $G$ be a graph following the random model \eqref{eq:random_model} and, given a set of nodes $S\subseteq V$, let $c_S = \max_{i\in S}c_i$ be the centrality of the most important node in $S$.  Then, for any node $i\notin S$ and any positive integer $n$, we have
\begin{align*}
 \mathbb P(\dist_G(i,S)> n) \geq \Big(1 - \binom{N-|S|}{n} s_\alpha(c_S-c_i) \Big)^{n}\, .
\end{align*}
\end{theorem}
\begin{proof}
First note that, for any $x, y \in \mathbb R$, it holds $s_\alpha(x)s_\alpha(y)\leq s_\alpha(x+y)$. Now let $i$ be a node such that $i\notin S$ and let   $\pi_n(i,S)$ be the probability that there exists in $G$ a walk of length $n$ from $i$ to the set $S$. Clearly it holds  $\mathbb P(\dist_G(i,S)=n) \leq \pi_n(i,S)$, moreover, a not difficult computation shows that  
\begin{align*}
 \pi_n(i,S) \leq  \binom{N-|S|}{n} s_\alpha(c_S-c_i)\, . 
\end{align*} 
 It follows that the probability  $\mathbb P(\dist_G(i,S)> n)$  that a node $i\notin S$ is at least $n$ steps far from the set $S$ can be lower-bounded as follows
 \begin{align*}
     \mathbb P(\dist_G(i,S)> n) &= \mathbb P(\, \overline{\{\cup_{k=1}^n \dist_G(i,S)=k\}}\, ) \\
     &\geq \prod_{k=1}^n(1-\mathbb P(\dist_G(i,S)=k)) \geq  \prod_{k = 1}^{n} \Big(1-\pi_k(i,S)\Big)\, 
 \end{align*}
 where  $\overline \Omega$ denotes the complement of $\Omega$. 
\end{proof}

Suppose for simplicity that the set of perturbed edges in $\tilde G$  is a clique $\delta E=S\times S$. If the size of $S$ is small enough and $c_i\gg c_S$, that is the node $i$ is much  more relevant than any node in $S$, then Theorem \ref{pro:prob_bound} shows that the probability that $i$   is $n$ steps far from $S$ is large. We shall show in the forthcoming Section \ref{sec:main:results} that the absolute variation $|f(A)_{k\ell}-f(\tilde A)_{k\ell}|$ decays exponentially with $\dist_G$. Thus, as claimed, in a model where the edge distribution follows the preferential attachment law \eqref{eq:random_model},  it is expected that changes in the topology of edges involving nodes with small centrality do not affect the ranking of the leading nodes or edges.

\section{Faber polynomials}\label{sec:faber}
In this section we review  the definition of the Faber polynomials  and several of their fundamental properties. 
This family of polynomials extends the theory of power series to sets different from the disk 
and will be used in the next section for our main results.

Given a continuum with connected complement $\F$, let us consider
the relative conformal map $\phi$ which maps the exterior of $\F$ onto the exterior of the unitary disk $\{z\in \mathbb C: |z|>1\}$, satisfying the following conditions
$$ 
\phi(\infty) = \infty, \quad \lim_{z \rightarrow \infty} \frac{\phi(z)}{z} =  d > 0.
$$
Hence, $\phi$ can be expressed by a Laurent expansion
$\phi(z) = dz + a_0 + \frac{a_1}{z} + \frac{a_2}{z^2} + \dots$. 
Furthermore, for every $n >0$ we have
$$
 \left ( \phi(z) \right)^n = 
    d^n z^n + a_{n-1}^{(n)} z^{n-1} + \dots + a_{0}^{(n)} + \frac{a_{-1}^{(n)}}{z} + \frac{a_{-2}^{(n)}}{z^2} + \dots \, .
$$
The Faber polynomial for the domain $\F$ is defined by
(see, e.g., \cite{suetin})
$$ 
\Phi_n(z) = d z^n + a_{n-1}^{(n)} z^{n-1} + \dots + a_{0}^{(n)}, \quad \textrm{ for } n \geq 0. 
$$
If $f$ is analytic on $\F$ then it can be expanded in a series
of Faber polynomials over $\F$, namely

\begin{theorem}[\cite{suetin}] \label{thm:faber} Let $f$ be analytic on $\F$. Let $\phi$ be the conformal map sending the exterior of $\F$ onto the exterior of the unitary disk. Let  $\psi$ be the inverse of $\phi$ and let and $\Phi_j$ be the $j$-th Faber polynomial associated with $\phi$. Then
 \begin{equation}\label{eq:faber:expansion}
   f(z) = \sum_{j = 0}^\infty f_j \Phi_j(z), \quad \textrm{ for } z \in \F;
\end{equation}
with the coefficients $f_j$ being defined by
$$
 f_j = \frac{1}{2 \pi i} \int_{D} \frac{f(\psi(z))}{z^{\,j + 1}} \, \textrm{d}z \, ,
$$
where $D$ is the boundary of a neighborhood of the unit disc such that $f$ in $\Omega$ can be represented in terms of its Cauchy integral on $\psi(D)$.
\end{theorem}

It immediately follows from the above theorem that, if the spectrum of $A$ is contained in $\F$ and $f$ is a function analytic in $\F$, then
the matrix function $f(A)$ can be expanded as follows
(see, e.g., \cite[p.~272]{suetin})
\begin{equation}\label{eq.fA.expansion}
f(A) = \sum_{j = 0}^\infty f_j \Phi_j(A) .
\end{equation}

The field of values or numerical range of a matrix $A\in \mathbb C^{N\times N}$ is a convex and compact subset of $\mathbb C$ defined by
$$\mathcal F(A) = \{x^* A x : x \in \mathbb C^{N}, \, \|x\|_2=1\} \, .$$

The following theorem, proved by Beckermann in \cite[Theorem 1]{beckermann_2005}, will be particularly useful in the following section.
\begin{theorem}\label{thm:beckermann} 
Let $A$ be a square matrix and let $\F$ a convex set containing $\mathcal F(A)$. Then for every $n \geq 1$ it holds
   $$ \| \Phi_n(A) \| \leq 2\, , $$
    $\Phi_n$ being the $n$-th Faber polynomial for the domain $\F$, as previously defined.
\end{theorem}

\section{Main results}\label{sec:main:results}
Consider a function $f:\mathbb C\to\mathbb C$ and let  $k,\ell$ be two nodes in $V$.  Assume that the adjacency matrix $A$ of $G=(V,E, \omega)$ is modified into the matrix $\tilde A = A+\delta A$ with associated graph $\tilde G$. As we discussed above,
we are interested in a-priori estimations of the absolute variation of the entries of $f(\tilde A)$ with respect to those of $f(A)$. 
To this end, in the  following Sections \ref{sec:main-bounds} and \ref{sec:stochastic-bounds}, we develop a number of explicit bounds of the form
\begin{equation}\label{eq:general_bound}
 |f(A)_{k\ell}-f(\tilde A)_{k\ell}| \leq \beta(\delta) \left( \frac{1}{M} \right)^{g(\delta)}, 
\end{equation}
where $\delta$ is a quantity that measures the distance in $G$ from $k$ to the set of modified edges  plus the distance from that set to $\ell$, $g(\delta)$ is an increasing linear function of $\delta$, 
$\beta(\delta) \rightarrow \beta > 0$ for $\delta \rightarrow +\infty$,
and, depending on the function $f$, $M$ 
may be constant or may vary with $\delta$ as well.

It is worthwhile pointing out that the bounds we propose depend on the distances between nodes in $G$  and the field of values of $A$ and $\tilde A$, whereas no knowledge on the topology of $\tilde G$ is required. This is particularly important as it allows us to formulate a simple algorithm that exploits the computations needed for computing the $f$-centrality or $f$-communities scores to simultaneously compute (or approximate) the distances between node pairs in $G$ and thus, for each node $k$ or pair of nodes $k$ and $\ell$, identifying via \eqref{eq:general_bound} the subgraphs of $G$ whose change in the edge topology do not affect (or affect in minor part) the score $f(A)_{k\ell}$.

\subsection{Upper bounds on network indices' stability}\label{sec:main-bounds}
In order to derive our bounds for the stability of $f(A)_{k\ell}$ we employ the theory of Faber polynomials briefly discussed in Section \ref{sec:faber}. This family of polynomials have been used for the analysis of the decay of the elements of functions of banded non-Hermitian matrices  \cite{benzi_boito_2014, benzi_razouk}. Based on those results and following the developments given in \cite{pozza.simoncini.2016} we derive our main Theorem \ref{thm.decay.tau} and its corollaries. To this end, we first need the following lemmas.
\begin{lemma}\label{lemma:Apow}
   Let $G=(V,E,\omega)$ be a graph and $A \in \mathbb{R}_{+}^{N\times N}$ be its adjacency matrix. 
   Consider the graph $\tilde G$, with adjacency matrix $\tilde A$,
   obtained by adding, erasing, or modifying the weights of the edges contained in $\delta E \subset V \times V$.
   If $S=\{s | (s,t) \in \delta E\}$ and $T=\{t | (s,t) \in \delta E\}$ are
   respectively the sets of sources and tips of $\delta E$, then 
   $$ ( p_n(\tilde A) )_{k \ell}  =  ( p_n(A) )_{k \ell}, $$
   for every polynomial $p_n$ of degree $n \leq \dist_G(k,S) + \dist_G(T,\ell) $.   
\end{lemma}
\begin{proof}
 We prove it for the monomials $( \tilde A^n )_{k \ell}  =  ( A^n )_{k \ell}$ concluding then by linearity.
 Since $( A^n )_{k \ell}$ is the weighted number of walks from $k$ to $\ell$ of length $n$,
 $( \tilde A^n )_{k \ell}  =  ( A^n )_{k \ell}$
 whenever $G$ and $\tilde G$ have the same walks of length $n$ from $k$ to $\ell$.
 Furthermore, a modified walk from $k$ to $\ell$ in $\tilde G$ must contain at least an edge from $S$ to $T$.
 We conclude noticing that any walk from $k$ to $\ell$ passing through $S$ and $T$ has length greater or equal to $\dist_G(k,S) + \dist_G(T,\ell) + 1$.
 \end{proof}

\begin{lemma}\label{thm.faber.sum}
   Consider the assumptions of Lemma \ref{lemma:Apow}.
   Moreover, let $\F$ be a convex continuum containing $\mathcal F(A)$ and $\mathcal F(\tilde A)$.
   If $f$ is an analytic function on $\F$ 
   and $f(z) = \sum_{j = 0}^\infty f_j \Phi_j(z)$ is its Faber expansion (\ref{eq:faber:expansion}) for the domain $\Omega$,
   then 
   $$\left| \left( f(A) -  f(\tilde A) \right)_{k \ell} \right | \leq 4 \sum_{j = \delta+1}^\infty |f_j|, $$
   where $\delta = \dist_G(k,S) + \dist_G(T,\ell) $.
\end{lemma}
\begin{proof}
  Since the coefficients $f_j$ depend only on the set $\F$ and the function $f$
  they are the same for both the expansions \eqref{eq.fA.expansion} of $f(A)$ and $f(\tilde A)$.
  Therefore, \eqref{eq.fA.expansion} gives
  $$f(A) -  f(\tilde A) =  \sum_{j = 0}^\infty f_j (\Phi_j(A) - \Phi_j(\tilde A)).$$
  By Lemma \ref{lemma:Apow} we get
  $$ \left( \Phi_j(A) - \Phi_j(\tilde A) \right)_{k \ell} = 0 \quad \textrm{ for every } j \leq \, \delta.$$
  Thus 
  $$\left( f(A) -  f(\tilde A) \right)_{k \ell} =  \sum_{j = \delta+1}^\infty f_j \left(\Phi_j(A) - \Phi_j(\tilde A)\right)_{k \ell}.$$
  Since $\F$ is convex, Theorem \ref{thm:beckermann} concludes the proof. 
\end{proof}

The previous lemma allows us to obtain the claimed exponential decay bound on the absolute variation of the entries of $f(A)$ and $f(\tilde A)$.

\begin{theorem}\label{thm.decay.tau}
Let $\Omega$ be a convex continuum containing $\mathcal F(A)$ and $\mathcal F(\tilde A)$, and let $\phi$ and $\psi$ be as in the assumptions of Theorem \ref{thm:faber}.
Given $\tau > 1$, if $f$ is analytic in the level set $G_\tau$ 
defined as the complement of the set  $\{ \psi(z) : |z| > \tau \}$,
   then 
   $$ 
      \left| \left( f(A) -  f(\tilde A) \right)_{k \ell} \right | \leq 
      \mu_\tau(f) \frac 2 \pi   \frac{\tau}{\tau -1} \left ( \frac{1}{\tau} \right)^{\delta + 2},
   $$
   where  $\delta =  \dist_G(k,S) + \dist_G(T,\ell)$ and 
$$
\mu_\tau(f) = \int_{D_\tau}|f(\psi(z))|\, \mathrm{d}z\, ,
$$
with $D_\tau = \{ z: |z| = \tau\}$.
\end{theorem}
\begin{proof} 
By Lemma~\ref{thm.faber.sum} we get 
$\displaystyle \left| \left( f(A) -  f(\tilde A) \right)_{k \ell} \right | \leq 4 \sum_{j = \delta+1}^\infty |f_j|$.
 The Faber coefficients $f_j$ are given by 
$$
 f_j = \frac{1}{2 \pi i} \int_{D_\tau} \frac{f(\psi(z))}{z^{j + 1}} \, \textrm{d}z.
$$
Thus $|f_j| \leq \frac 1 {2\pi \tau^{j+1}} \mu_\tau(f)$, 
and
\begin{eqnarray*}
     \left| \left( f(A) -  f(\tilde A) \right)_{k \ell} \right |
	     &\leq&	
\mu_\tau(f) \frac 2 \pi   \sum_{j = \delta + 1}^\infty \left ( \frac{1}{\tau} \right)^{j+1} \\
			       & = &    
\mu_\tau(f) \frac 2 \pi \left ( \frac{1}{\tau} \right)^{\delta + 2} \sum_{j = 0}^\infty \left ( \frac{1}{\tau} \right)^j \\
			       & = &    
\mu_\tau(f) \frac 2 \pi \frac{\tau}{\tau - 1}  \left ( \frac{1}{\tau} \right)^{\delta + 2} \end{eqnarray*}
concluding the proof.
\end{proof}

The theorem above shows that 
if $k$ is distant from $S$, or $T$ is distant from $\ell$,
then $f(\tilde A)_{k\ell}$ is close to $f(A)_{k\ell}$.
Moreover, the difference between the two values decreases exponentially in $\delta = \dist_G(k,S) + \dist_G(T,\ell)$.
As a limit case, if the considered graph is not connected and there is no walk either from $k$ to $m$ or from $n$ to $\ell$,
then  $\delta = +\infty$ and we deduce $f(\tilde A)_{k\ell} = f(A)_{k\ell}$.

In order to obtain a sharp bound in Theorem~\ref{thm.decay.tau} 
we need to choose $\tau$ appropriately.
This choice clearly depends on the trade-off between $\mu_\tau(f)$, that is the possibly ``large size'' of $f$ on the given
region, and the exponential decay of $(1/\tau)^{\delta +2}$.
Hence, Theorem \ref{thm.decay.tau} produces a family of bounds depending on $\tau$, $f$ and the fields of values $\mathcal F(A), \mathcal F(\widetilde{A})$. 
	
As we discussed in Section \ref{sec:network-things}, the exponential and the resolvent functions \eqref{eq:exp_and_res} play a  central role for $f$-centrality and $f$-communicability problems in the complex networks literature \cite{benzi2013ranking,estrada2010network,estrada2005subgraph}. For this reason in what follows we focus on these two functions and derive more precise bounds when $f(x)$ is either $\exp(x)$ or $r_\alpha(x)$.

We will use the symbol $\Re(z)$ to denote the real part of the complex number $z$. 
\begin{corollary}\label{cor.exp}
   Let $A, \tilde A, S$ and $T$ be as in Lemma \ref{lemma:Apow},
   $\F$ be a set containing $\mathcal F(A)$ and $\mathcal F(\tilde A)$,
   and $\delta = \dist_G(k,S) + \dist_G(T,\ell)$.
   
   If the boundary of $\F$ is a horizontal ellipse with semi-axes  $a \geq b > 0$ and center $c$,
   and  $\delta > b - 1$ then
$$ 
   \left| \left( \exp(A) -  \exp(\tilde A) \right)_{k \ell} \right |     \leq 
 		 \frac{4 e^{\Re(c)} p(\delta +1)}{p(\delta + 1) - (a+b)/(\delta + 1)} 
			\left( \frac{a+b}{\delta + 1}\frac{e^{q(\delta+1)}}{p(\delta+1)}\right)^{\delta+1}, 			
   $$ 
    with 
    $q(t) = 1 + \frac{a^2 - b^2}{t^2 + t\sqrt{t^2 + a^2 - b^2}}$ and 
    $p(t) = 1 + \sqrt{1 + (a^2 - b^2)/t^2} $.
    Moreover, for the special cases where $\F$ is either a disk or a line segment, we have:
  \begin{enumerate}  
\item    If $\F$ is a disk of radius  $a$ and center $c$,
   and  $\delta  > a - 1$ then
   $$ 
   \left| \left( \exp(A) -  \exp(\tilde A) \right)_{k \ell} \right |     \leq 
 		4 e^{\Re(c)}  \frac{(\delta + 1)}{\delta + 1 - a} 
			\left( \frac{a e}{\delta + 1}\right)^{\delta+1}.			
   $$ 
   
\item    If $\F$ is a real interval $[c-a,c+a]$ (with $a>0$),
    then for every  $\delta > 0$
   \begin{equation*}
   \left| \left( \exp(A) -  \exp(\tilde A) \right)_{k \ell} \right |     \leq 
 		  \frac{4 e^{c} p(\delta +1)}{p(\delta + 1) - a/(\delta + 1)} 
			\left( \frac{a}{\delta + 1}\frac{e^{q(\delta+1)}}{p(\delta+1)}\right)^{\delta+1}, 			
\end{equation*}
    with 
    $q(t) = 1 + \frac{a^2}{t^2 + t\sqrt{t^2 + a^2}}$ and 
    $p(t) = 1 + \sqrt{1 + (a/t)^2} $.
\end{enumerate}

\end{corollary}
\noindent Notice that for $t$ large enough $p(t) \approx 2$ and $q(t) \approx 1$.
\begin{proof}
 We begin with the case of  $\F$ with boundary an horizontal ellipse.
 A conformal map for $\F$ is 
 \begin{equation}\label{eq:phi:ellipse}
   \phi(w) = \frac{w - c - \sqrt{(w-c)^2 - \rho^2 }}{\rho R},
 \end{equation}
 and its inverse is
 \begin{equation}\label{eqn-psi-ellipse}
    \psi(z) = \frac{\rho}{2}\left ( Rz + \frac{1}{Rz} \right) + c, 
 \end{equation}
with $\rho = \sqrt{a^2 - b^2}$ and $R = (a + b)/\rho$;
see, e.g., \cite[chapter II, Example 3]{suetin}.
Notice that
 $$ 
\max_{|z| = \tau} |e^{\psi(z)}| = \max_{|z| = \tau} e^{\Re(\psi(z))} = 
e^{\frac{\rho}{2}\left ( R\tau + \frac{1}{R\tau} \right) + \Re{(c)}}.
$$ 
For every $\tau > 1$, by Theorem \ref{thm.decay.tau}, we get
$$ \left| \left( \exp(A) -  \exp(\tilde A) \right)_{k \ell} \right |     \leq 
			4 \frac{\tau}{\tau - 1} e^{\Re{(c)}} e^{\frac{\rho}{2}\left ( R\tau + \frac{1}{R\tau} \right)} \left( \frac{1}{\tau} \right)^{\delta + 1},$$
where we used $\mu_\tau(\exp) \leq 2 \pi \tau\max_{|z|=\tau}|e^{\psi(z)}|$.
The optimal value of $\tau$ which minimizes
 $ e^{\frac{\rho}{2}\left ( R\tau + \frac{1}{R\tau} \right)} \left( \frac{1}{\tau} \right)^{\delta + 1}$ is
$$\tau_* = \frac{\delta + 1 + \sqrt{(\delta + 1)^2 + \rho^2 }}{\rho R} \,,$$
and we have $\tau_*>1$. In fact, since $\delta + 1 > \frac{\rho}{2}\left(R - \frac{1}{R} \right) = b$ holds by assumption, we get 
$$
\tau_* = \frac{\delta + 1 + \sqrt{(\delta + 1)^2 + \rho^2 }}{\rho R} > \frac{b + \sqrt{b^2 + \rho^2 }}{\rho R} = \frac{b+a}{\rho R}=1\, .
$$
Finally, noticing that
$$ \frac{\rho}{2}\left ( R\tau_* + \frac{1}{R\tau_*} \right) = (\delta + 1) q(\delta + 1),  $$
the proof is completed for the ellipse case. 

The case in which $\F$ is a disk is easily obtained by setting $b = a$, 
while the case $\F = [c -a, c + a]$ can be proved considering an ellipse of center $c$, major axis $a$, minor axis any $b > 0$,
and then letting $b \rightarrow 0$ in the bound for the ellipse case.
\end{proof}

Similarly, we can derive a bound for the resolvent $r_\alpha(A) = (I - \alpha A)^{-1}$.
In this case, the function $r_\alpha(z)$ is not analytic in the whole complex plane.
This property has crucial effects in the approximation, 
as the subsequent corollary shows.
\begin{corollary}\label{cor.res}
   Let $A, \tilde A, S$ and $T$ be as in Lemma \ref{lemma:Apow},
   $\F$ be a set symmetric with respect to the real axis and  containing $\mathcal F(A)$ and $\mathcal F(\tilde A)$,
   $\delta = \dist_G(k,S) + \dist_G(T,\ell)$,
   and $r_\alpha(x)$ be defined as in \eqref{eq:exp_and_res}
   with $\alpha >0$ such that $\alpha^{-1} \notin \F$.
   
   If the boundary of $\F$ is a horizontal ellipse with semi-axes  $a \geq b > 0$ and center $c$,
   then for $0 < \varepsilon < |\alpha^{-1} - c| - a$ and  $\delta > 0$ it holds
   $$ 
     \left| \left( r_\alpha(A) -  r_\alpha(\tilde A) \right)_{k\ell} \right |     \leq
 		\frac{4 }{1  - \frac{a+b}{(|\alpha^{-1} - c|  - \varepsilon)p_\varepsilon}}  \frac{1}{\varepsilon}
			\left( \frac{a+b}{|\alpha^{-1} - c|  - \varepsilon} \frac{1}{p_\varepsilon} \right)^{\delta + 1},
   $$ 
    where $p_\varepsilon = 1 + \sqrt{1 - (a^2 - b^2)/(|\alpha^{-1} - c|  - \varepsilon)^2}$.
    Moreover, for the special cases where $\F$ is either a disk or a line segment, we have:
    \begin{enumerate}
    \item If $\F$ is a disk of radius  $a$ and center $c$,
    then for $0 < \varepsilon < |\alpha^{-1} - c| - a$  and  $\delta > 0$ it holds
       $$ 
     \left| \left( r_\alpha(A) -  r_\alpha(\tilde A) \right)_{k\ell} \right |     \leq
 		\frac{4 }{1  - \frac{a}{(|\alpha^{-1} - c|  - \varepsilon)}}  \frac{1}{\varepsilon}
			\left( \frac{a}{|\alpha^{-1} - c|  - \varepsilon} \right)^{\delta + 1} .
   $$ 

    \item If $\F$ is a real subinterval $[c-a,c+a]$ (with $a>0$),
    then for $0 < \varepsilon < |\alpha^{-1} - c| - a$  and  $\delta > 0$ it holds
    \begin{equation*}
     \left| \left( r_\alpha(A) -  r_\alpha(\tilde A) \right)_{k\ell} \right |     \leq
 		\frac{4 }{1  - \frac{a}{(|\alpha^{-1} - c|  - \varepsilon)p_\varepsilon}}  \frac{1}{\varepsilon}
			\left( \frac{a}{|\alpha^{-1} - c|  - \varepsilon} \frac{1}{p_\varepsilon} \right)^{\delta + 1},
\end{equation*}
    where $p_\varepsilon = 1 + \sqrt{1 - (a/(|\alpha^{-1} - c|  - \varepsilon))^2}$.
    \end{enumerate}
\end{corollary}

 Notice that since adjacency matrices are real their field of values are symmetric with respect to the real axis and
hence the assumption on $\F$ is natural.
We also remark that $p_\varepsilon \approx 2$ when $\delta$ is large.
\begin{proof}
 Here we prove the case of $\F$ with ellipse shape
 since the other two cases can then be derived
 as done in the proof of Corollary \ref{cor.exp}.
 
Let $\phi$ as in \eqref{eq:phi:ellipse} and $\psi$ as in \eqref{eqn-psi-ellipse}.
 Since the function $(1 - \alpha z)^{-1}$ is not analytic in $\alpha^{-1}$ 
 in order to fulfill the assumptions of Theorem \ref{thm.decay.tau} we assume
 $|\psi(z)| < \alpha^{-1}$ for every 
$|z| > 1$. 
 
 Notice that  $\frac{\rho}{2}\left ( R\tau + \frac{1}{R\tau} \right) $ is the major semi-axis of 
 the ellipse $\Gamma_\tau = \{\psi(z) \, : \, |z|= \tau \}$.
 Since the center of the ellipse is on the real axis, 
 for $\varepsilon > 0$ small enough we get
 $$ 
{\varepsilon} =
		    \min_{|z| = \tau_\varepsilon} \left|\, \psi(z) - \alpha^{-1} \, \right|
		  =  |\alpha^{-1} - c| - \frac{\rho}{2}\left ( R\tau_\varepsilon + \frac{1}{R\tau_\varepsilon} \right) ,
$$
with 
$$ \tau_\varepsilon = \frac{|\alpha^{-1} - c|  - \varepsilon + \sqrt{(|\alpha^{-1} - c| - \varepsilon)^2 - \rho^2}}{\rho R}. $$
Notice that the condition $\tau_\varepsilon >1$ is satisfied if and only if $\varepsilon < |\alpha^{-1} - c| - a$, which holds by assumption.
 Hence by Theorem \ref{thm.decay.tau} we get
   $$ 
 \left| \left( r_\alpha(A) - r_\alpha(\tilde A) \right)_{k \ell} \right |     \leq
      4 \frac{\tau_\varepsilon}{\tau_\varepsilon - 1} \frac{1}{\varepsilon} \left ( \frac{1}{\tau_\varepsilon} \right)^{\delta + 1},
   $$
which gives the bound.
\end{proof}

\subsection{Normalized adjacency matrices: Random walks on $G$}\label{sec:stochastic-bounds}
In many cases the adjacency matrix of a graph $G=(V,E,\omega)$ is ``normalized'' into a transition matrix, so to model a random walk process on the edges. 
Transition matrices (or random walks matrices) arise in many network applications, including centrality, quasi-randomness and clustering problems (e.g. \ \cite{brandes2005network, chung2008quasi,rohe2011spectral, von2007tutorial}). 

Assume for simplicity that the graph $G$ is unweighted, loop-free and with no dangling nodes. That is, if $A$ is the adjacency matrix of $G$, then $a_{ij}\in \{0,1\}$, $a_{ii}=0$ $\forall i, j$ and, for each $i$, there exists at least one $j$ such that $a_{ij}=1$.
A popular transition matrix $\mathcal A_{out}$ on $G$ describes the stationary random walk on the graph where   a walker standing on a vertex $i$ 
chooses to walk along one of the outgoing edges of node $i$, with no preference among such edges. The entries of $\mathcal A_{out}$ are the probabilities of going from node $i$ to node $j$ in one step, which are then given by $(\mathcal A_{out})_{ij} = (D_{out}^{-1}A)_{ij}$, where $A$ is the adjacency matrix of $G$,  $D_{out}=\diag(d_1^{out} , \dots, d_n^{out})$ and $d_i^{out}$ is the number of outgoing edges from node $i$. 
Note that when the graph is undirected  the adjacency matrix $A$ is symmetric, however the transition matrix is not. This is one of the reasons why a symmetrized version of $\mathcal A_{out}$ is typically preferred  in this case. 
Such matrix, defined by $\mathcal A = D_{out}^{1/2}\mathcal A_{out}D_{out}^{-1/2}=D_{out}^{-1/2}AD_{out}^{-1/2}$, is also known as normalized adjacency matrix of $G$. 

In this section we discuss how the bounds of Section \ref{sec:main-bounds} transfer to $\mathcal A_{out}$ and $\mathcal A$. For the sake of simplicity, let us first address the undirected  case. 

For a set of nodes $S\subseteq V$ let $\partial_1(S) =\{i\in V\setminus S : \dist_G(S,i)=1\}$ denote the neighborhood of $S$. Let $\mathcal A$ and $\tilde {\mathcal A}$ be the normalized adjacency matrices of $G$ and  $\tilde G$, respectively. Unlike the conventional adjacency matrix, the set of entries that are affected by the changes in $E$ are not only related to $\delta E$, but to a larger set. Precisely,  if the edges in $\delta E$ connect the nodes within $S \subset V$, then changes in $\mathcal A$ occur on the entries corresponding to the nodes in $\bar S = S \cup \partial_1(S)$. Note that, given $k\notin S$, we have  
\begin{equation}\label{eq:dist:scaled}
   \dist_G(k,S) =  \dist_G(k, \bar S) + 1. 
\end{equation}
Therefore, an easy consequence of Lemma \ref{lemma:Apow} applied to $\bar S$ implies that Theorem \ref{thm.decay.tau} holds for $|f(\mathcal A)_{k\ell}-f(\tilde {\mathcal A})_{k\ell}|$ when $\delta$ is replaced by  $\dist_G(k,S) + \dist_G(S, \ell)-2$. However a more careful analysis of the structure reveals that the following lemma holds: 
\begin{lemma}\label{lemma:Apow:normalized}
   Let $G=(V,E,\omega)$ be an undirected graph and $\mathcal A \in \mathbb{R}_{+}^{N\times N}$ be its normalized adjacency matrix. 
   Consider the graph $\tilde G$, 
   obtained by adding, erasing or modifying the weight of the edges between the nodes in a subset $S$ and let $\tilde {\mathcal A}$ be the corresponding normalized adjacency matrix. 
   Then, for any  $k, \ell \notin S$ we have
   $$ ( p_n(\tilde{\mathcal  A}) )_{k \ell}  =  ( p_n(\mathcal A) )_{k \ell}\, , $$ 
   for every polynomial $p_n$ of degree $n \leq \dist_G(k,S) + \dist_G(S,\ell) - 1$.   
\end{lemma}
\begin{proof}
  A walk from $k$ to $\ell$ in $\tilde G$ contains a modified edge
  only if it passes through at least one modified edge in $S$ or 
  through at least one re-weighted edge connecting  $S$ and $\partial_1(S)$.
  Therefore, any modified walk must go from $k$ to $\partial_1(S)$, then from $\partial_1(S)$ to $S$,
  then from $S$ to $\partial_1(S)$, and finally from $\partial_1(S)$ to $\ell$.
  Therefore, due to the identity \eqref{eq:dist:scaled}, the length of any modified walk must be equal or longer than 
  $$\dist_G(k, \bar S) + \dist_G(\bar S, \ell) + 2 = \dist_G(k, S) + \dist_G(S, \ell). $$
  The proof thus follows as the one of Lemma \ref{lemma:Apow}.
\end{proof}

Hence, following the same arguments as the one in Section \ref{sec:main:results}, we can extend to $\mathcal A$ the bounds of Theorem \ref{thm.decay.tau}, Corollary \ref{cor.exp} and Corollary \ref{cor.res} by replacing $\delta$ with $\delta - 1$. 

Let us now consider the case of a directed network and let us thus transfer Lemma \ref{lemma:Apow} to the transition matrix $\mathcal A_{out}$. Arguing as above we obtain
\begin{lemma}\label{lemma:Apow:transition}
  Let $G=(V,E,\omega)$ be a directed graph and $\mathcal A_{out} \in \mathbb{R}_{+}^{N\times N}$ be the its transition matrix. 
   Consider the graph $\tilde G$, 
   obtained by adding or erasing the edges in $\delta E \subset V \times V$, and let $\tilde {\mathcal A}_{out}$ be the corresponding transition matrix.
   If $S=\{s | (s,t) \in \delta E\}$ and $T=\{t | (s,t) \in \delta E\}$ are
   respectively the sets of sources and tips of $\delta E$, then 
   $$ ( p_n(\tilde {\mathcal A}_{out}) )_{k \ell}  =  ( p_n(\mathcal A_{out}) )_{k \ell}, $$ 
   for every polynomial $p_n$ of degree $n \leq \dist_G(k,S) + \min\{\dist_G(T,\ell), \dist_G(S,\ell) -1 \}$. 
\end{lemma}
\begin{proof}
  Consider $\partial_1^{out}(S) =\{i\in V\setminus S : \dist_G(S,i)=1\}$,
  the out-neighborhood of $S$.
  A walk from $k$ to $\ell$ in $\tilde G$ contains a modified edge
  only if it passes through at least one modified edge in $S$ or 
  through at least one re-weighted edge connecting  $S$ and $\partial_1^{out}(S)$.
  Therefore, any modified walk must go from $k$ to $S$, then it may go from $S$ to $T$ through a modified edge,
  or it can go to any node in $\partial_1^{out}(S)$.
  In the first case, the length from $k$ to $\ell$ of the walk must be greater or equal than
  $$ \dist_G(k,S) + \dist_G(T,\ell) + 1. $$
  In the second case, the length of the walk must be greater or equal than
  $$ \dist_G(k,S) + \dist_G(S,\ell). $$
  The proof thus follows as the one of Lemma \ref{lemma:Apow}.
\end{proof}
Hence, 
we can extend to $\mathcal A_{out}$ the bounds of Theorem \ref{thm.decay.tau}, Corollary \ref{cor.exp} and Corollary \ref{cor.res} 
by replacing $\delta$ with $\min\{\delta, \dist_G(k,S) + \dist_G(S,\ell) -1\}$.

\subsection{On the field of values of adjacency matrices} 
Two quantities play a key role in the computation of the bounds we proposed: the shortest-path distances between pairs of nodes and the shape of the field of values of $A$ and $\tilde A$. The next subsection deals with the former whereas we  devote the present subsection to the latter.

Let $A$ be an $N\times N$ matrix.  The spectral and numerical radius of $A$ are the quantities 
$$\rho(A)=\max\{|\lambda| : \lambda \text{ is eigenvalue of }A\}, \qquad \nu(A) = \max\{|w| : w \in \mathcal F(A)\}\, , $$ 
respectively. 
When $A$ has nonnegative entries ($A\geq 0$), the Perron-Frobenius theorem ensures that $\rho(A)$ belongs to the spectrum of $A$ and addresses various properties of such eigenvalue and the associated eigenvectors.  For completeness, we recall part of that theorem here below (see for instance \cite{berman1994nonnegative})
\begin{theorem}[Perron-Frobenius theorem]
 Let $A\geq 0$. Then 
$\rho(A)$ is an eigenvalue of $A$ and there exist  $x,y \geq 0$ such that $Ax=\rho(A)x$, $y^T A = \rho(A)y^T$. Moreover, $\rho(A)\geq \rho(B)$ for any $B\geq 0$ such that $A\geq B$ entry-wise. 
If $A\geq 0$ is irreducible, then 
$\rho(A)$ is simple and nonzero and the corresponding left and right eigenvectors $x,y$ are positive and unique (up to a scalar multiple). Moreover, $\rho(A)>\rho(B)$ for any $0\leq B\leq A$ such that $B\neq A$.
\end{theorem}
As for the spectrum of $A$, a Perron-Frobenius theory for the field of values $\mathcal F(A)$ has been developed in relatively recent years. We collect in the theorem below a number of results, borrowed from \cite{li2002numerical,maroulas2002perron},  which are  useful for our scopes. To this end, let us further define  the Hermitian part of $A$ as the Hermitian matrix $H_A=(A+A^*)/2$. 
\begin{theorem}[\cite{li2002numerical,maroulas2002perron}]\label{thm:PF_for_F(A)}
Let $A\geq 0$. Then
\begin{enumerate}
 \item $\nu(A)$ is the element of maximal modulus in $\mathcal F(A)$, attained by the Perron eigenvector of $H_A$. Precisely, we have 
\begin{equation}\label{eq:numerical_radius}
 \nu(A)=\rho(H_A) = x^* A x\, 
\end{equation}
where $x$ is such that $x\geq 0$, $\|x\|_2=1$ and $H_Ax=\rho(H_A)x$. 
\item  If $A$ is irreducible, then the maximal elements of $\mathcal F(A)$ are 
$$\nu(A)\exp(2i\pi p/k), \qquad p=1,2,\dots, k-1$$
being $k$ the index of imprimitivity of $A$, i.e.\ the number of eigenvalues of $A$ with maximal modulus.
\end{enumerate}
\end{theorem}
Point 1 of Theorem \ref{thm:PF_for_F(A)} shows that, for nonnegative matrices $A\geq 0$, it is always possible to compute a set $\F$ containing the field of values of $A$ or of $\tilde A$, by letting $\F = \{\zeta\in \mathbb C : |\zeta|\leq \nu\}$ where $\nu$ is $\nu(A)$ or $\nu(\tilde A)$, respectively. 
Moreover,  for irreducible nonnegative matrices $A$, point  2 in Theorem \ref{thm:PF_for_F(A)} shows how the shape of the field of values $\mathcal F(A)$ is  characterized in terms of the index of imprimitivity of $A$. Thus, if the imprimitivity index of $A$ is large, then  the ball $\{\zeta\in \mathbb C : |\zeta|\leq \nu(A)\}$ is a tight approximation of the field of values $\mathcal F(A)$.



The normalized adjacency matrix $\mathcal{A}$ has the desirable property of being diagonally similar to a stochastic matrix. This implies that $\nu(\mathcal A)=\nu(\tilde{\mathcal A})=1$. For general nonnegative matrices, instead, the field of values can be large. However, due to \eqref{eq:numerical_radius}, the numerical radii  $\nu(A)$ and $\nu(\tilde A)$ can be well approximated by standard eigenvalues techniques such as the power method or the Lanczos process. The computational cost of this operation is much smaller than the one required to compute the entries of the matrix functions $f(A)$ or $f(\tilde A)$. Moreover, if $\delta A$ is sparse enough, we expect $\nu(A)$ and $\nu(\tilde A)$ to be close. This claim is also supported by the following Theorem \ref{thm:bound_on_r(A)}, where the case of a single-entry perturbation is discussed.   
\color{blue}
\begin{theorem}\label{thm:bound_on_r(A)}
 Let $A\geq 0$ and let $\tilde A = A + \varepsilon \uno_m\uno_n^T$, with $\varepsilon>0$ and $m\neq n$. Then
 \begin{enumerate}
  \item $0\leq \nu(\tilde A)-\nu(A) \leq \varepsilon/2$ and, if $H_{\tilde A}$ is irreducible, then $\nu(\tilde A)-\nu(A)>0$.
  \item Assume $A$ is symmetric  and irreducible. Let  $\lambda_i$ be the eigenvalues of $A$ different from $\rho(A)$ and let $\mathrm{gap} = \min_i|\lambda_i-\rho(A)|$.  If $\varepsilon \leq 2\,\mathrm{gap}$ then for any nonnegative function $f$ 
  such that $f(\nu(A))\neq 0$, we have
  $$0<\nu(\tilde A)-\nu(A) \leq \varepsilon\frac{\sqrt{f(A)_{mm}f(A)_{nn}}}{f(\nu(A))} + O(\varepsilon^2)$$
 \end{enumerate}
\end{theorem}
\begin{proof}
 As $\tilde A \geq A \geq 0$, then $H_{\tilde A} \geq H_A \geq 0$ and, by the Perron-Frobenius theorem and point 1 of Theorem \ref{thm:PF_for_F(A)}, we have $\nu(\tilde A)\geq \nu(A)$. 
Now let $\delta A = \varepsilon\uno_m\uno_n^T$. We have $H_{\delta A}(\uno_m\pm \uno_n)=\pm (\uno_m\pm \uno_n)\varepsilon/2$. Therefore, as the rank of $H_{\delta A}$ is at most 2, we have $\rho(H_{\delta A})=\|H_{\delta A}\|_2=\varepsilon/2$. Moreover,  as $H_A$ is symmetric, there exists an orthogonal basis of eigenvectors $V$ with $\|V\|_2\|V^{-1}\|_2 =1$. Thus, by the Bauer-Fike theorem (see, e.g.,  \cite{golub2012matrix}) applied to $H_{\tilde A} = H_A+H_{\delta A}$ we get 
 $$\nu(\tilde A) - \nu(A)=\rho(H_{\tilde A}) - \rho(H_A)\leq \| H_{\tilde A}-H_{A}\|_2 = \|H_{\delta A}\|_2=\varepsilon/2\, .$$
 Assume now that $H_{\tilde A}$ is irreducible. As $\tilde A \neq A$ we have  $H_{\tilde A}\neq H_A$ and, by the Perron-Frobenius theorem applied to $H_A$ and $H_{\tilde A}$, we deduce the strict inequality $\nu(\tilde A)>\nu(A)$.  This completes the proof of the first point. 

 To address the second statement note that, as  $A$ is irreducible and symmetric, 
  $\nu(A)=\rho(A)$ is a simple eigenvalue and the corresponding eigenvector $x$ with $x^Tx=1$ is entry-wise positive. By Theorem \ref{thm:PF_for_F(A)}, $\nu(\tilde A)$ is a positive eigenvalue of $H_{\tilde A} = H_A + H_{\delta A}= A + H_{\delta A}$ that coincides with its spectral radius. Let $M=(\uno_m\uno_n^T+\uno_n\uno_m^T)/2$. If we write $H_{\tilde A}=A(\varepsilon) = A+\varepsilon M$ as a function of $\varepsilon$ then, by a standard eigenvalue perturbation argument (see, e.g., \cite{wilkinson1965algebraic}), there exist smooth curves $\lambda(\varepsilon)$ and $x(\varepsilon)$ such that $A(\varepsilon)x(\varepsilon)=\lambda(\varepsilon)x(\varepsilon)$ and $\lambda(0)=\nu(A)$, $x(0)=x$. If we differentiate the eigenvalue equation for $\lambda(\varepsilon)$ and we evaluate it at $\varepsilon=0$ we obtain 
  $$
  \left(M-\lambda'(0)I \right)x + \left(A-\nu(A)I \right)x'(0)=0 
  $$ 
  and, by multiplying this identity on the left by $x^T$, we get $\lambda'(0)=x^TMx=x_mx_n$. This, together with 
  $$
  \lambda(\varepsilon) = \lambda(0)+\lambda'(0)\varepsilon+O(\varepsilon^2)\, ,
  $$
  implies that $\nu(A)$ is perturbed into
 \begin{equation*}\label{eq:a}
  \lambda(\varepsilon)= \nu(A) + \varepsilon x_mx_n +O(\varepsilon^2)\, . 
 \end{equation*}
 Now, let $y_1, \dots , y_{n-1}, x$ be the orthonormal eigenvectors of $A$ corresponding to the eigenvalues $\lambda_1 \leq \dots \leq \lambda_{n-1} < \lambda_n = \nu(A)$. Then, as $f$ is nonnegative, for any $1\leq i\leq n$, we have 
 $$
 f(A)_{ii}=f(\nu(A))x_i^2 + \sum_{j=1}^{n-1}f(\lambda_j)(y_j^T \uno_i)^2\geq f(\nu(A))x_i^2\, ,
 $$
with $\uno_i$ being the $i$-th canonical vector.  
 As a consequence $x_i\leq  \sqrt{f(A)_{ii}/f(\nu(A))}$ and, in particular, 
 \begin{equation*}\label{eq:x_i}
 x_mx_n\leq \frac{\sqrt{f(A)_{mm}f(A)_{nn}}}{f(\nu(A))} \, .
 \end{equation*}
  To conclude the proof we observe that $\nu(\tilde A) =\lambda(\varepsilon)$.

Let us order the eigenvalues of $H_{\tilde A}=A+H_{\delta A}$ 
as $\mu_1 \leq \cdots \leq \mu_n = \nu(\tilde A)$. 
For any $k$, let $U_k$ be the $k$-dimensional subspace spanned by the eigenvectors of $\lambda_1, \dots, \lambda_k$. Then
$$
\mu_k = \min_{\mathrm{dim}\, U =\,k}\max_{y\in U}\frac{y^T(A+H_{\delta A})y}{y^T y} \leq \max_{y\in U_k}\frac{y^T(A+H_{\delta A})y}{y^T y} \leq \lambda_k + \|H_{\delta A}\|_2 .
$$
Therefore any eigenvalue $\mu_k$ of $H_{\tilde A}$ is upper-bounded by $\lambda_k + \varepsilon/2$. 
%
Since $\mathrm{gap} \geq \varepsilon/2 $ by assumption we deduce that $\mu_{n-1}\leq \lambda_n= \nu(A)$. Together with the inequality $\nu(\tilde A)>\nu(A)$ shown by the first point of this theorem we obtain $\nu(\tilde A)=\lambda(\varepsilon)$, as claimed. 
\end{proof}
\color{black}

\subsection{Computing node distances by Krylov methods}\label{sec:lanczos:dist}
The bounds proposed so far rely on the geodesic distances  between pair of nodes in the graph. Computing such distances is a classical problem in graph theory and several efficient and parallelizable algorithms are available \cite{thomas2001introduction}. In this section, however, we propose a simple numerical strategy that exploits the computation of $f(A)_{k\ell}$,  to simultaneously approximate the   distances $\dist_G(m,\ell)$ and $\dist_G(k,m)$ for any node $m$, at essentially no additional cost. 
As we will discuss in what follows, this procedure is well suited for undirected graphs and allows to compute small distances exactly, whereas it provides a lower bound when $\dist_G(m,\ell)$ (resp.\ $\dist_G(k,m)$) is too large.  

Computing the $f$-centrality or the $f$-communicability of nodes and edges of a network can be a computationally expensive task, especially for large graphs. However, when only few entries of $f(A)$ are needed, an established and  efficient strategy to address these quantities exploits the fact that  $f(A)_{k\ell}$ can be written as the quadratic form $\uno_k^T f(A)\uno_\ell$ and  thus employs Lanczos-type algorithms \cite{lanczos:50,lanczos:52} both for symmetric \cite{benzi:boito:2010, golub_meurant_2009} and non-symmetric matrices \cite{fenu:2013, FreHoc93, PozPraStr16}. 

Given a vector $b$ and a matrix $A$, the Krylov subspace $\mathcal K_n(A,b)$ is the vector space spanned by $\{b, Ab, \dots, A^{n-1}b\}$.
The non-Hermitian Lanczos algorithm produces two bases $\{v_0,\dots,v_{n-1}\}$ and $\{w_0,\dots,w_{n-1}\}$ for  $\mathcal{K}_n(A,v_0)$ and $\mathcal{K}_n(A^T,w_0)$, respectively.
If no breakdowns arise, the $j$-th vectors $v_j$ and $w_j$ are obtained at the  $j$-th step of the algorithm.
Moreover, they are biorthogonal ($w_j^*v_j = 1$) and such that
$$ v_j = p_j(A)v_0 \quad \text{ and } \quad w_j = \bar{p}_j(A^*)w_0, $$
with $p_j$ a polynomial of degree exactly $j$ and
$\bar{p}_j$ the polynomial with conjugated coefficients.
We remark that the Hermitian Lanczos algorithm can be used when $A$ is symmetric and $v_0 = w_0$. For symmetric matrices and the case $v_0\neq w_0$, a similar strategy can be employed (see \cite[Chapter\ 7]{golub_meurant_2009} for details). 
In the following we treat only the non-Hermitian Lanczos algorithm. Everything can be easily transferred to the Hermitian case by letting $w_j = v_j$. 

In order to approximate $f(A)_{k\ell}$ the method requires to set $v_0 = \uno_\ell$ and $w_0 = \uno_k$.
We then get the following result:
\begin{theorem}
    Let $A\geq 0$ be the $N \times N$ adjacency matrix of $G=(V,E,\omega)$ and let $\{v_0,\dots,v_{n-1}\}$ and $\{w_0,\dots,w_{n-1}\}$ be the basis of $\mathcal{K}_n(A,\uno_\ell)$ and $\mathcal{K}_n(A^T,\uno_k)$, respectively, 
   obtained by the non-Hermitian Lanczos algorithm. 
   Then for every $m=1,\dots,N$ the distance $\dist_G(m,\ell)$ (resp.\ $\dist_G(k,m)$) is equal to the first index $j$ for which the $m$-th element of $v_j$ (resp.\ $w_j$) is nonzero.
\end{theorem}
\begin{proof}
\begin{rev}
We prove the result for $\dist_G(m,\ell)$. 
Moreover, since the case $m=\ell$ is trivial (the $\ell$-th element of $v_0$ is nonzero) we assume $m\neq \ell$.
Since $(A^j)_{m, \ell}$ is the overall weight of the walks of length $j$ from $m$ to $\ell$, we get $(p_j(A))_{m,\ell} = 0$ for every $ j < \dist_G(m,\ell)$. 
Finally, if $j = \dist_G(m,\ell)$, then $(A^{i})_{m, \ell} = 0$ for $i=0,\dots,j-1$ and $(A^{j})_{m, \ell} \neq 0$. Hence, there exists $\alpha\neq 0$ such that  $(p_j(A))_{m,\ell} = \alpha (A^j)_{m, \ell} \neq 0$. 
\end{rev}
\end{proof}

Therefore, we can modify the non-Hermitian Lanczos algorithms to compute the distance vectors 
$$ \overrightarrow{d}_k = \left ( \begin{array}{c}
                           \dist_G(k,1) \\
                           \vdots \\
                           \dist_G(k,N)
                          \end{array}
 \right) , 
 \qquad \overleftarrow{d}_\ell = \left ( \begin{array}{c}
                           \dist_G(1,\ell) \\
                           \vdots \\
                           \dist_G(N,\ell)
                          \end{array}
 \right). $$
The idea is to add the following pseudo-code to the Lanczos algorithm (we assume to stop it at the $(n-1)$-th iteration).

First, initialize the variables \verb is_zero_k , \verb is_zero_l , \verb d_k  and \verb d_l   as follows

\begin{samepage}
\noindent\rule[.0em]{\textwidth}{.5pt}
\begin{verbatim}
  for i=1,...,N
    is_zero_k(i) = TRUE
    is_zero_l(i) = TRUE
    d_k(i) = n   % vector of the distances from k
    d_l(i) = n   % vector of the distances to l
  end
  is_zero_k(k) = FALSE
  is_zero_l(l) = FALSE
  d_k(k) = 0   
  d_l(l) = 0   
\end{verbatim}
\noindent\rule[.7em]{\textwidth}{.5pt}
\end{samepage}
\noindent and then  modify the method by adding the procedure below,  to derive the distances from the nonzero pattern of $v_j$ and $w_j$, at each step of the scheme

\begin{samepage}
\noindent\rule[.0em]{\textwidth}{.5pt}
\begin{verbatim}
  for j=1,...,n-1  % Iteration of Lanczos algorithm
    compute vector v_j and  w_j
    for m=1,...,N
      if v_j(m) > 0 && is_zero_k(m)
        d_k(m) = j
        is_zero_k(m) = FALSE
      end
      if w_j(m) > 0 && is_zero_L(m)
        d_l(m) = j
        is_zero_l(m) = FALSE
      end
    end
    proceed with the rest of the algorithm
  end         
\end{verbatim}
\noindent\rule[.7em]{\textwidth}{.5pt}
\end{samepage}

Notice that if $n$ is smaller or equal than the diameter of the graph, there can be null elements in $v_j$ for $j=0,\dots,n-1$.
Nevertheless, for all these elements, $n$ is a lower bound for the distance, which can then be used  to  approximate $\delta$ in Theorem \ref{thm.decay.tau}.  On the other hand, let us remark that many real-world networks have a small diameter, thus  we expect the proposed technique to be able to actually compute the desired distances in typical applications. Also note that by using this strategy, computing the diagonal of $f(A)$ allows to simultaneously address the all-pair shortest-path distances of the graph. This is particularly effective when dealing with undirected graphs. In that case, in fact, the entries $f(A)_{kk}$ can be computed with the Hermitian Lanczos method which further ensures no breakdowns. 
%
%
%
%
Finally, note that the proposed modified Lanczos method relies on the the transition of the variables \verb+v_j(m)+ and \verb+w_j(m)+ from zero to nonzero. This is not an issue as no round-off error can arise in the computation of \verb+v_j(m)+ (resp.\ \verb+w_j(m)+) until $j<\dist_G(k,m)$. This is because the entry \verb+v_j(m)+ is always given by the sum of products of quantities being exactly zero. Hence, numerical bias may affect the method only if, when $j = \dist_G(k,m)$, the value of \verb+v_j(m)+ is smaller than the machine precision. 
%
%
%
%
%
Table \ref{tab:dist} in the next section shows how this strategy  behaves on four example undirected networks, where the diagonal of the exponential function is approximated by the Hermitian Lanczos method and the number of maximal iterations varies.

\section{Numerical examples}\label{sec:experiments}
In this section we illustrate the behavior of the proposed bounds on some example networks.

The first explanatory example graph we consider is represented in  Figure \ref{fig:2cycles}, where we provide a pictorial representation of the graph before and after the perturbation and where we highlight important quantities such as the sets $S$ and $T$ and the distances from and to such sets. 
The considered graph $G$ is made by two simple cycles (closed undirected paths) with $111$ nodes each,
and by one directed edge $\overrightarrow{e}$ from node $111$ to node $112$. 

\begin{figure}[htb]
\definecolor{trueblue2}{RGB}{0, 115, 230} \definecolor{magentix}{RGB}{233,76,111}
\begin{tikzpicture}[scale=.35, baseline={(1,0)}]
\node at (-3,-3) {$G$};
\draw[thick]  (0,0) circle (3cm);
\node[circle, draw=black, thick, fill=white,  scale=.8, inner sep=1] (111) at (10*360/10: 3cm) {\footnotesize $111$};
\node[circle, draw=black, thick, fill=white,  scale=.8, inner sep=1] (110) at (360/10: 3cm)    {\footnotesize $110$};
\node[minimum size=.58cm,circle, draw=black, thick, fill=white,  scale=.8, inner sep=1] (1) at (9*360/10: 3cm)  {\footnotesize $1$};
\node[minimum size=.58cm,circle, draw=black,thick,fill=white,scale=.8, inner sep=1] (2) at (8*360/10: 3cm)  {\footnotesize $2$};
\node[minimum size=.58cm, circle, draw=black, thick, fill=white,  scale=.8, inner sep=1] (k) at (4*360/10: 3cm) {\footnotesize $k$};
\draw[dashed,draw=white,very thick] (4.3*360/10:3cm) arc (4.3*360/10:7.7*360/10:3cm);
\draw[dashed,draw=white,very thick] (1.35*360/10:3cm) arc (1.35*360/10:3.7*360/10:3cm);
\draw[thick]  (10,0) circle (3cm);
\node[circle,draw=black,thick,fill=white,scale=.8,inner sep=1] (112) at  ([shift={(10,0)}] 180:3cm) {\footnotesize $112$};
\node[circle,draw=black,thick,fill=white,scale=.8,inner sep=1] (113) at ([shift={(10,0)}] 180-360/10: 3cm){\footnotesize $113$};
\node[circle,draw=black, thick,fill=white,  scale=.8,inner sep=1] (114) at ([shift={(10,0)}] 180-360/5: 3cm){\footnotesize $114$};
\node[circle,draw=black, thick,fill=white,  scale=.8,inner sep=1] (222) at ([shift={(10,0)}] 180+360/10:3cm){\footnotesize $222$};
\draw[dashed,draw=white,very thick,shift={(10,0)}] (193+360/10:3cm) arc (193+360/10:180+7.7*360/10:3cm);
\draw[thick,-{>[scale=1.5]}] (111) to [out=30,in=150] node[below]{$\overset{\rightarrow}{e}$} (112);
\draw[thick,draw=trueblue2,-{>[scale=1.5]}] (10*360/10:2cm) arc (360: 360+4*360/10: 2cm);
\draw[thick,draw=trueblue2,-{>[scale=1.5]}] (360+4*360/10:3.9cm) arc (360+4*360/10: 375: 3.9cm) to [out=20,in=120] (112);
\node (d1) at (4*360/10:1.1cm) {\color{trueblue2} \footnotesize $\dist_G(T,k)$};
\node (d2) at (2*360/15:5cm) {\color{trueblue2} \footnotesize $\dist_G(k,S)$};
\end{tikzpicture}\hfill
\begin{tikzpicture}[scale=.35, baseline={(1,0)}]
\node at (-3,-3) {$\tilde G$};
\draw[thick]  (0,0) circle (3cm);
\node[circle, draw=black, thick, fill=white,  scale=.8, inner sep=1] (111) at (10*360/10: 3cm) {\footnotesize $111$};
\node[circle, draw=black, thick, fill=white,  scale=.8, inner sep=1] (110) at (360/10: 3cm)    {\footnotesize $110$};
\node[minimum size=.58cm,circle, draw=black, thick, fill=white,  scale=.8, inner sep=1] (1) at (9*360/10: 3cm)  {\footnotesize $1$};
\node[minimum size=.58cm,circle, draw=black,thick,fill=white,scale=.8, inner sep=1] (2) at (8*360/10: 3cm)  {\footnotesize $2$};
\draw[dashed,draw=white,very thick] (1.35*360/10:3cm) arc (1.35*360/10:7.7*360/10:3cm);

\draw[thick]  (10,0) circle (3cm);
\node[circle,draw=black,thick,fill=white,scale=.8,inner sep=1] (112) at  ([shift={(10,0)}] 180:3cm) {\footnotesize $112$};
\node[circle,draw=black,thick,fill=white,scale=.8,inner sep=1] (113) at ([shift={(10,0)}] 180-360/10: 3cm){\footnotesize $113$};
\node[circle,draw=black, thick,fill=white,  scale=.8,inner sep=1] (114) at ([shift={(10,0)}] 180-360/5: 3cm){\footnotesize $114$};
\node[circle,draw=black, thick,fill=white,  scale=.8,inner sep=1] (222) at ([shift={(10,0)}] 180+360/10:3cm){\footnotesize $222$};
\draw[dashed,draw=white,very thick,shift={(10,0)}] (193+360/10:3cm) arc (193+360/10:180+7.7*360/10:3cm);
\draw[thick,-{>[scale=1.5]}] (111) to [out=30,in=150] (112);
\draw[thick,-{>[scale=1.5]},draw=magentix] (112) to [out=210,in=330] node[below]{\color{magentix}$\overset{\leftarrow}{e}$} (111);
\draw[thick,yellow,fill=yellow,opacity=0.2] (111.center) ellipse (1.3cm and 1cm);
\draw[thick,yellow,fill=yellow,opacity=0.2] (112.center) ellipse (1.3cm and 1cm);
\node (T) at ([shift={(-1.5,0)}] 111.center)  {$T$}; 
\node (S) at ([shift={(1.5,0)}] 112.center)  {$S$}; 
\end{tikzpicture}
\caption{This figure shows an example network $G$ (left) formed by two cycles connected by a directed ``bridge'' $\overset{\rightarrow}{e}$. To such network we add the edge $\overset{\leftarrow}{e}$ in the opposite direction (colored in red) obtaining $\tilde G$, on the right. We highlight important quantities arising in our work: $\dist_G(k,S)$ and $\dist_G(T,k)$ are shown in blue, whereas the sets $S=\{112\}$ and $T=\{111\}$ are colored  in yellow.}
\label{fig:2cycles}
\end{figure}
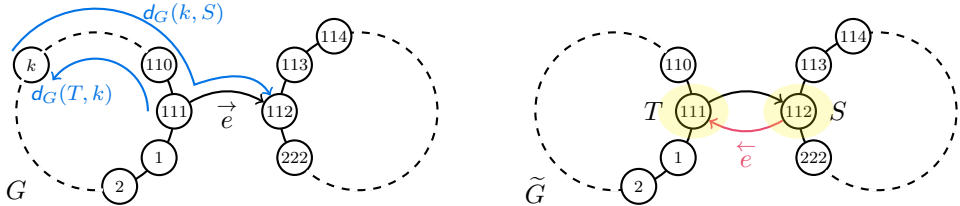

Since there are no closed walks passing through $\overrightarrow{e}$, all the nodes in $G$ have the same $f$-centrality. The graph is then perturbed by the insertion of one single new directed edge $\overleftarrow{e}$ from node $112$ to node $111$. This new edge ``closes the two-directional bridge''  between the two circles, resulting into a perturbation of the $f$-centrality scores of the nodes.  
In Figure \ref{fig:2cycles:bounds} 
we plot in red crosses the values $|\exp(A)_{kk} - \exp(\tilde A)_{kk} |$  (left plot) 
and those of $|r_\alpha(A)_{kk} - r_\alpha(\tilde A)_{kk} |$ with \note{
$\alpha = 1/3$ } (right plot),
for $k=1,\dots,222$.
With blue circles, instead, we represent the bound in Corollaries \ref{cor.exp} (left plot) and \ref{cor.res} (right plot)
for every admissible $k$. 
As we can see the behavior of the decays of the differences is well approximated by the bounds. 
Moreover, we observe that the exponential decay for the $\exp$-centrality variation as well the linear decay of the $r_\alpha$-centrality are captured by the bound.
%
\begin{figure}[h!]
 \centering
 \includegraphics[width=.45\textwidth, trim =  18mm 1mm 33mm 5mm, clip]{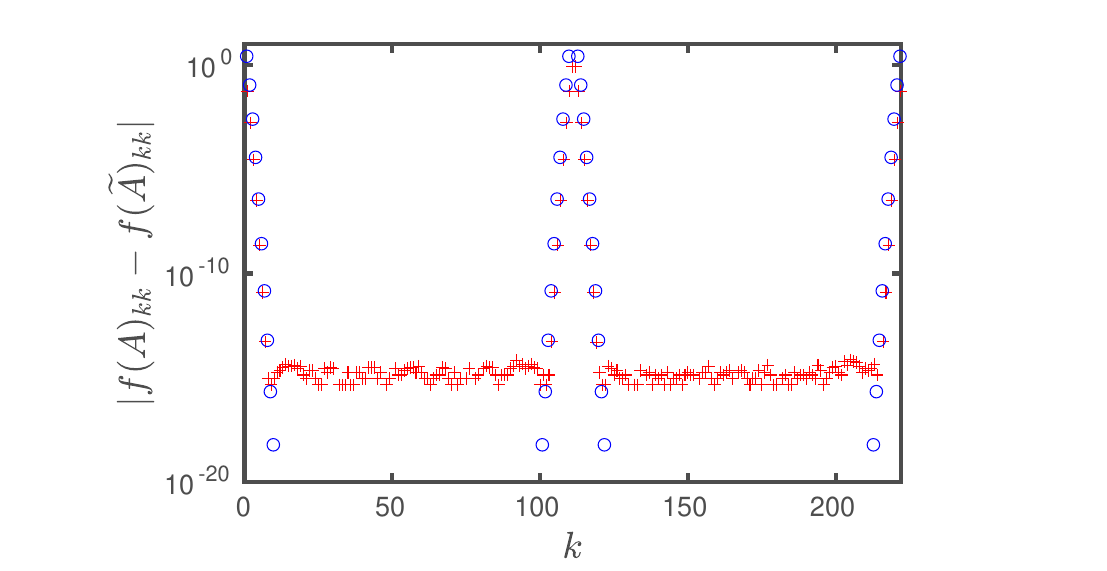}\hspace{1.5em}
 \includegraphics[width=.45\textwidth, trim =  18mm 1mm 33mm 4mm, clip]{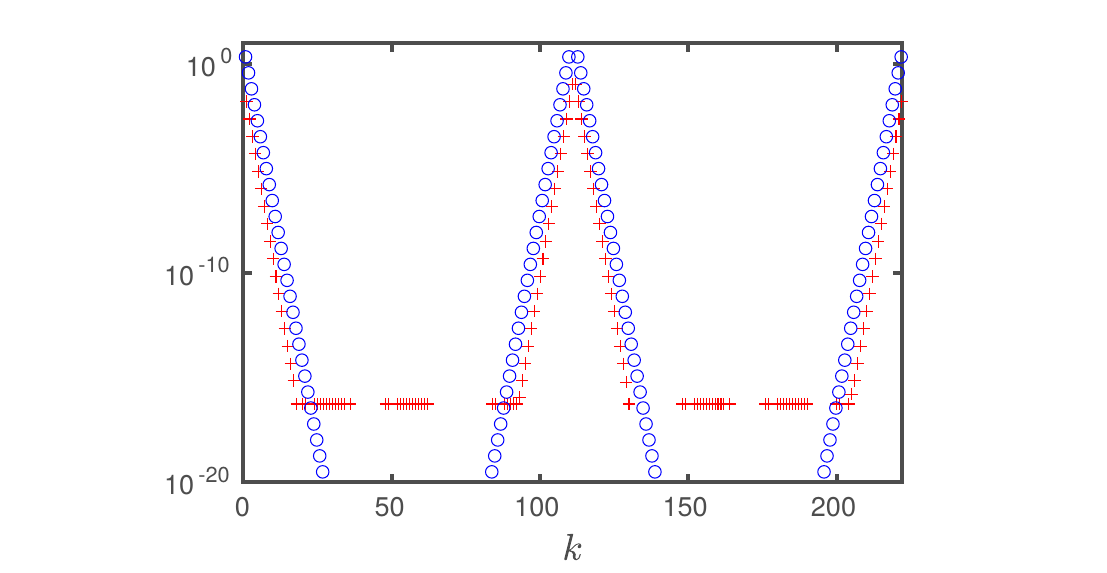}\hspace{.5em}
 \caption{In the ordinate: $|f(A)_{kk} - f(\tilde A)_{kk} |$ (red crosses) when adding the edge $\overset{\leftarrow}{e}$ in the graph of Figure \ref{fig:2cycles}
and the bounds (blue circles) of Corollaries \ref{cor.exp} and \ref{cor.res}.
In the abscissa: $k$, the nodes of the graph. From left to right: variation of exponential centrality $f(x)=\exp(x)$;  variation of resolvent centrality  \note{$f(x)=r_{1/3}(x)$.}}\label{fig:2cycles:bounds}
\end{figure}
 
The second example is borrowed from a real-world data set representing the London city transportation network~\cite{DDSBA14,DataDD}. The undirected weighted network that we consider here is the aggregate version of the original multi-layer network. The 369 nodes correspond to train stations and the existing routes between them are the edges. Each edge is weighted with an integer number that accounts for the number of train connections between the stations (overground, underground, DLR, etc.). A picture representing this network is shown in Figure \ref{fig:london_drawings} (left). The color of the edges is darker if the weight is higher and the size of the nodes  is proportional to the importance they are given by the $\exp$-centrality score.  We perform two experiments on this network with the goal of exploring both the roles of the weight function and of the edge topology in the proposed bounds. We subdivide the two cases into two subsections.

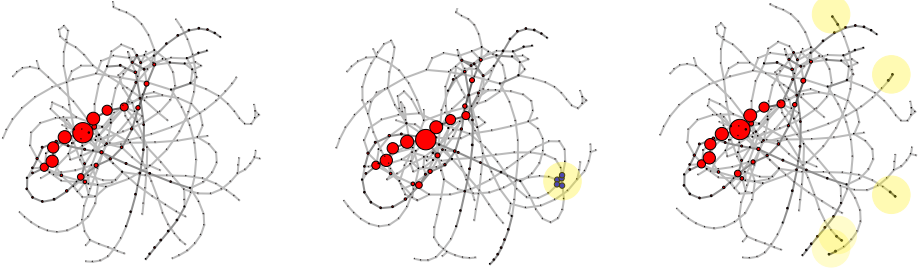
\begin{figure}[t]
    \begin{tikzpicture}[scale=.89]
    \node[anchor=south west,inner sep=0] at (0,0) {\includegraphics[width=.32\textwidth, trim = 2cm 1.5cm 2cm 1cm, clip]{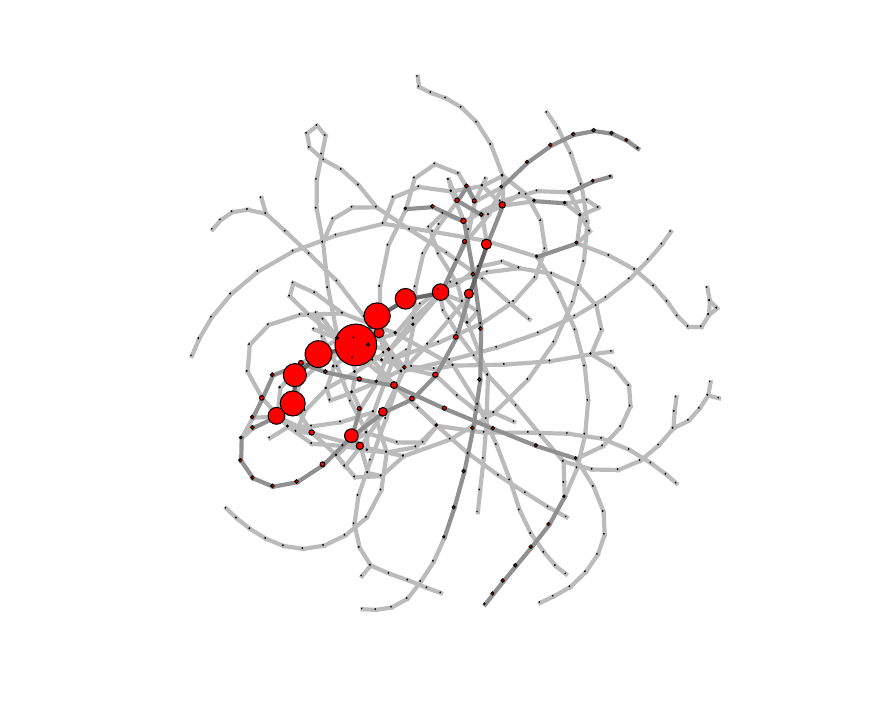}};
    \end{tikzpicture}
    \hfill
    \begin{tikzpicture}[scale=.89]
    \node[anchor=south west,inner sep=0] at (0,0) {\includegraphics[width=.32\textwidth, trim = 2cm 1.15cm 2cm 1cm, clip]{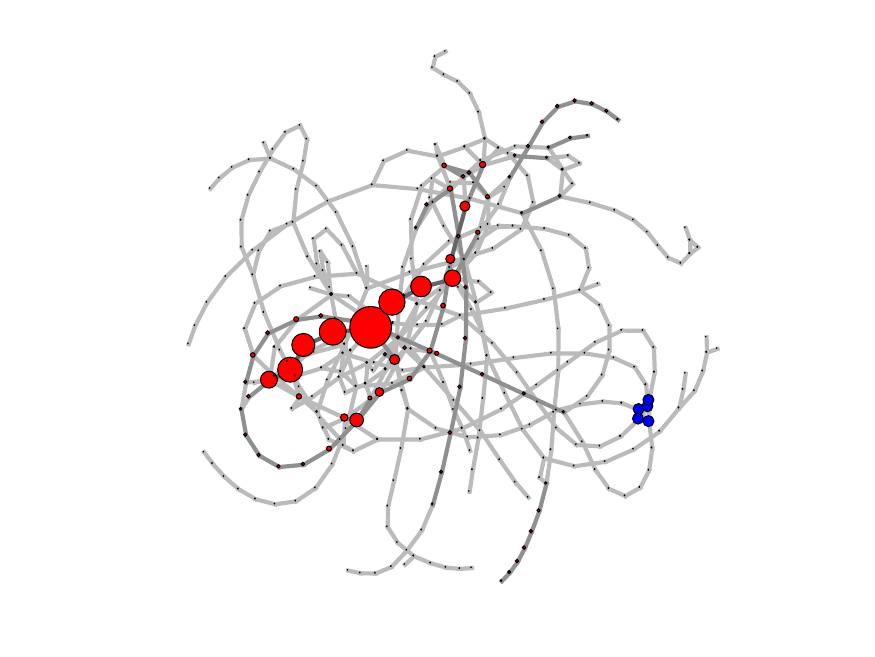}};
    \draw[draw=yellow, fill=yellow, opacity=0.3] (3.9,1.3) circle (.28); 
    \end{tikzpicture}
    \begin{tikzpicture}[scale=.89]
    \node[anchor=south west,inner sep=0] at (0,0) {\includegraphics[width=.32\textwidth, trim = 2cm 1.15cm 2cm 1cm, clip]{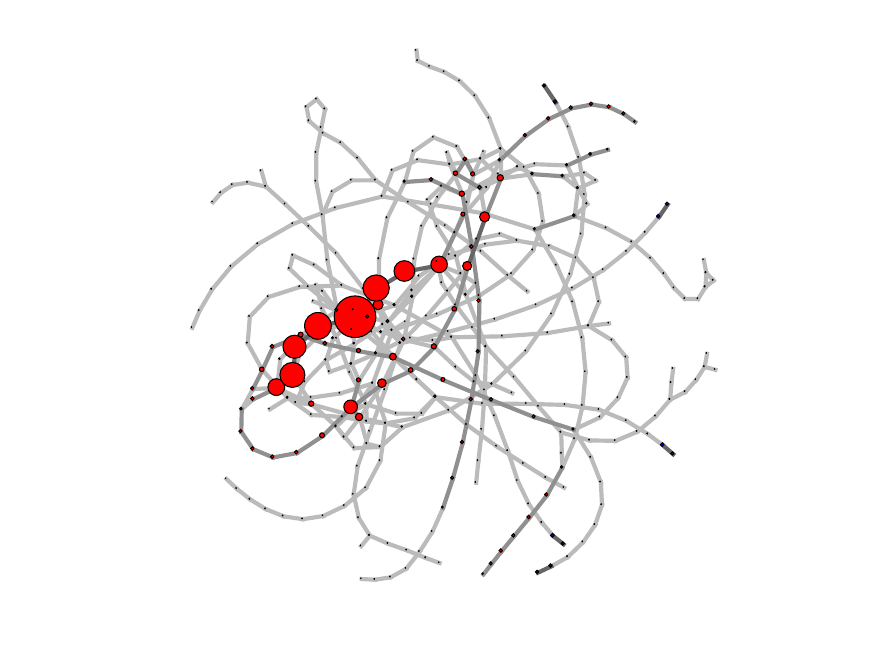}};
    \draw[draw=yellow, fill=yellow, opacity=0.3] (4,1) circle (.28); 
    \draw[draw=yellow, fill=yellow, opacity=0.2] (3.1,.2) circle (.28); 
    \draw[draw=yellow, fill=yellow, opacity=0.2] (3.2,.4) circle (.28); 
    \draw[draw=yellow, fill=yellow, opacity=0.3] (3.1,3.7) circle (.28); 
    \draw[draw=yellow, fill=yellow, opacity=0.3] (4,2.8) circle (.28); 
    \end{tikzpicture}
    \caption{Graph drawing of the weighted aggregate London city transportation multi-layer network. The color of the edges is darker proportionally to the corresponding weight. Left: Original network $G$; Center:  network $\tilde G$ obtained from $G$ by inserting a clique of edges connecting the 5 least ranked nodes of $G$, according to the $\exp$-centrality measure; Right: network $\tilde G$ obtained from $G$ by changing the weight of the edges involving the 5 least ranked nodes of $G$. Modified nodes and edges are highlighted with a yellow circle. The size of the nodes is proportional to their importance according to the $\exp$-centrality of the corresponding graph ($G$ and $\tilde G$, respectively). } \label{fig:london_drawings}
\end{figure}
\subsubsection*{Case 1: Adding new edges} In the first experiment we perturb the original network $G$ by fully  connecting the 5 nodes $\{i_1,\dots,i_5\}$ with smallest $\exp$-centrality. The graph drawing of this modified network $\tilde G$ is shown in Figure \ref{fig:london_drawings} (center). As the network is undirected and we add a fully connected subgraph,  the sets $S$ and $T$ both coincide with the nodes $\{i_1,\dots,i_5\}$. This set is highlighted with a yellow shape in the figure and the corresponding nodes are colored in blue (their size is proportional to their $\exp$-centrality score in $\tilde G$). This drawing shows that the topology of the network as well as the centrality of the top ranked nodes is very poorly affected by the addition of the new edges.  We show in Figure \ref{fig:london_plots_1} further evidence in this sense. The first plot (on the left) shows the {\it intersection similarity}~\cite{FKS03}  between the ranking  of the nodes (based on $\exp$-centrality) before and after the edge modification, whereas the second plot (on the right) shows, as before,  the comparison between the actual values of $|\exp(A)_{kk} - \exp(\tilde A)_{kk}|$  (red crosses)  and the bound given by Corollary \ref{cor.exp} (blue circles). Let us point out that in this plot (and all the following ones), we are relabeling the nodes according with the distance from (and to) the set of modified nodes. Precisely, if $\partial_k(S) =\{\ell\in V: \dist_G(\ell,S)=k\}$ with $|\partial_k(S)|=n_k$, then we label the nodes $V=\{1,\dots, N\}$ of $G$ so that
$$
S = \{1,\dots,5\}, \quad \partial_1(S) = \{6, \dots, n_1+5\}, \quad \partial_{2}(S) = \{n_1+6,\dots, n_1+n_2+5\}\quad \cdots
$$
 As it is reasonable to expect, the proposed  bounds are now less tight than the previous explanatory example of the two circles, but still the decay behavior is  well captured.

For the sake of completeness we recall here that the intersection similarity is a measure used to compare the top $\kappa$ entries of two ranked lists $\ell^{1}$ and $\ell^{2}$ that 
may not contain the same elements.  It is defined as follows: let $\ell^{j}_\kappa$ be the list of the top $\kappa$ elements in $\ell^{j}$, for 
$j=1,2$.  Then, the {\it top $\kappa$ intersection similarity between $\ell^{1}$ and $\ell^{2}$} is defined as 
\begin{equation*}
{\rm isim}_\kappa(\ell^{1},\ell^{2}) =1- \frac{1}{\kappa}\sum_{t=1}^\kappa\frac{|\ell^{1}_t\Delta\ell^{2}_t|}{2\, t},
\label{eq:isim}
\end{equation*}
where $|\ell^{1}_t\Delta\ell^{2}_t|$ denotes the number of elements in the  symmetric difference between $\ell^{1}_t$ and $\ell^{2}_t$. 
When the ordered sequences contained in $\ell^{1}$ and $\ell^{2}$ are completely different, then the intersection similarity between the two is minimum and it is equal to 0, whereas,  the intersection similarity between  $\ell^{1}$ and $\ell^{2}$  is equal to $1$ 
if and only if the two ordered sequences coincide.   
\begin{figure}[t]
    \includegraphics[width=.99\textwidth,trim=3.5cm 5cm 3.8cm 5.5cm,clip]{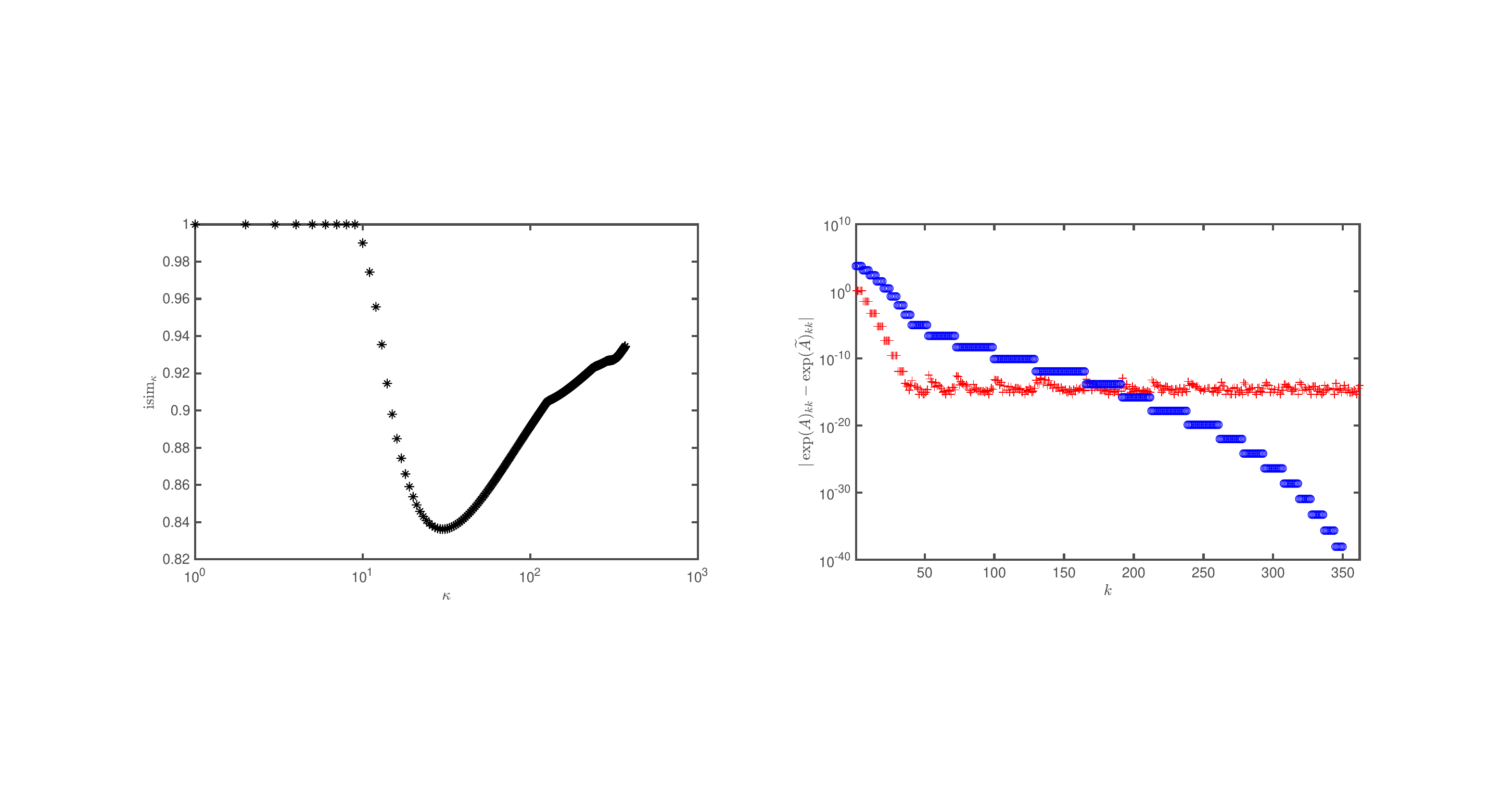}
    \caption{Intersection similarity and variation of $|\exp(A)_{kk} - \exp(\tilde A)_{kk}|$ for the weighted network of London transports considered in Figure \ref{fig:london_drawings} (left and central drawings). \textbf{Left:} Intersection similarity $\mathrm{isim_\kappa}(c_1,c_2)$ where $(c_1)_i = \exp(A)_{ii}$ and $(c_2)_i = \exp(\tilde A)_{ii}$ for $i=1, \dots, N$. \textbf{Right:} Red crosses are the actual difference $|\exp(A)_{kk} - \exp(\tilde A)_{kk}|$ whereas the blue dots are the bound of Corollary \ref{cor.exp}. }
    \label{fig:london_plots_1}
\end{figure}
\subsubsection*{Case 2: Modifying the edge weights} In the  second experiment we change the weight of the edges  that connect the $m$ nodes $\{i_1,\dots,i_m\}$ with least $\exp$-centrality in $G$, for $m\in\{5,15,30\}$. For each $i_j$ we look at the set of its neighbors $\partial_1(i_j) =\{\ell\in V: \dist_G(i_j,\ell)=1\}$ and define the set of perturbed nodes as $S=T=\cup_{j=1}^m \{i_j, \partial_1(i_j)\}$. Then, we modify the weight of each edge  $e$ connecting nodes in $S$  by letting $\tilde \omega(e)=5+\omega(e)$. Note that the sets $S$ and $T$ coincide once again, but they change depending on the choice of $m\in\{5,15,30\}$. In Figure \ref{fig:london_drawings} (right) we show a picture representing the case $m=5$. The set of perturbed nodes $S$ is highlighted with yellow disks. Again we see that the topology and the centrality of the top ranked nodes is very poorly affected in this case. Drawing of the case $m\in \{15,30\}$ are not shown for the sake of brevity.  In Figure \ref{fig:london_plots_2} we show how the actual   values of $|\exp(A)_{kk} - \exp(\tilde A)_{kk}|$  (red crosses)  and the bound given by Corollary \ref{cor.exp} (blue circles) change when the number of perturbed edges $m$ change. The plots there shown correspond (from left to right) to the case $m=5$, $m=15$ and $m=30$, respectively. The plots show evidence of the fact  that the decay behavior is well captured also in these cases. Recall that, as before, nodes in the figure are re-labeled according to the distance with respect to the set of modified nodes.  
\begin{figure}[t]
    \includegraphics[width=.99\textwidth,trim=3.2cm 7cm 3.8cm 7.2cm,clip]{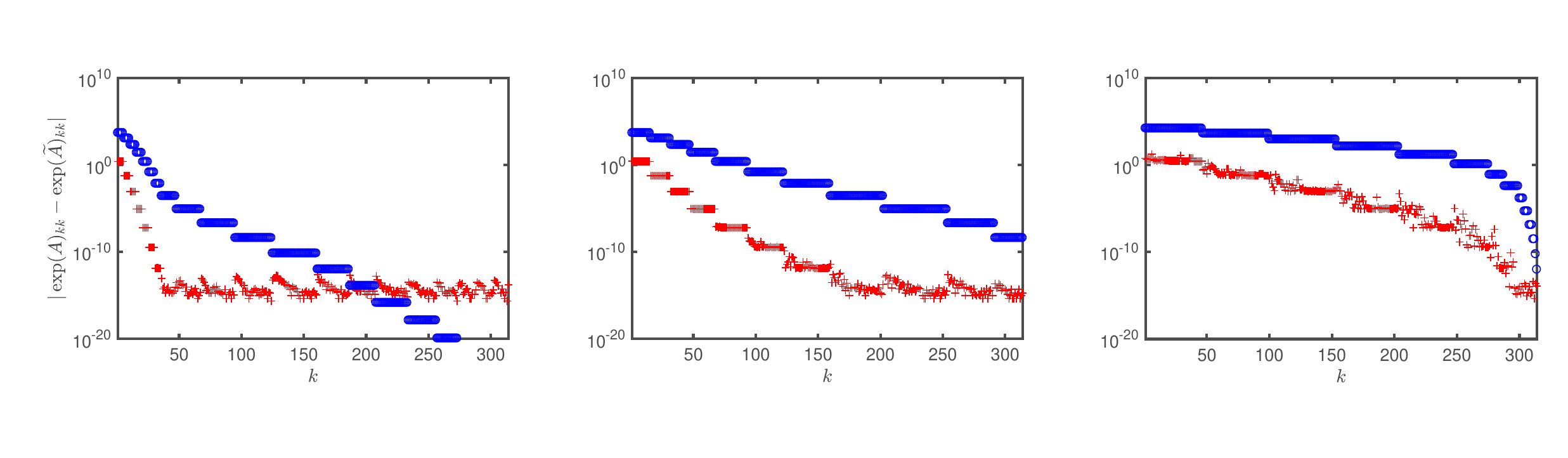}
    \caption{Variation of $|\exp(A)_{kk} - \exp(\tilde A)_{kk}|$ for the weighted network considered in Figure \ref{fig:london_drawings} (left and right drawings). The network is perturbed into $\tilde G$ where the weight of the edges  $e$ that connect the $m$ nodes  with least $\exp$-centrality in $G$ is modified into $\tilde \omega(e) = 5+\omega(e)$. From left to right, the plots show the case $m=5,15,30$, respectively.}\label{fig:london_plots_2} 
\end{figure}

\subsection{Small-world networks}
Finally, in the following we present four examples of larger scale  real-world undirected networks  whose diameter is proportional to the  logarithm of the number of nodes. This feature deserves attention as it is common to many complex networks and is related to the so called ``small-world'' phenomenon \cite{watts1998collective}.  
  We focus here on the analysis of the correlation between the variation of the network  centralities and the variation of the distances in $G$ with respect to the set of perturbed edges. For this reason, normalized adjacency matrices $\mathcal A$ are considered below, so to guarantee the  field of values of both the original and the perturbed matrices to be constrained within the unit segment $[-1,1]$. Thus, in all our experiments below we set $c=0$ and $a=1$ in the formula of  Corollary \ref{cor.exp} for the real interval.

The considered network data are borrowed from \cite{uniflor, snapnets} and are briefly described below:
 \begin{description}
  \item[Gnutella]  A snapshot of the Gnutella peer-to-peer file sharing network in August, 8, 2002. Nodes represent hosts in the Gnutella network topology and edges represent connections between the Gnutella hosts. Number of nodes: 6300, Number of edges: 41297, Diameter: 10;
  \item[Facebook] This dataset consists of anonymous ``friends circles'' from Facebook. Facebook data was collected from survey participants. Number of nodes: 4038, Number of edges: 176167, Diameter: 10;
  \item[GRCQ] General Relativity and Quantum Cosmology (GR-QC) collaboration network. Data are collected from the e-print arXiv and covers scientific collaborations between authors of papers submitted to GR-QC category. The data covers papers in the period from January 1993 to April 2003. Number of nodes: 5242, Number of Edges: 14496, Diameter: 17;
  \item[Erd\"os] Erd\"os collaboration network. Number of nodes: 472, Number of edges: 2628, Diameter: 11.
 \end{description}



We approximate the distances $\dist_G(k,\ell)$, $k, \ell = 1,\dots, N$, by running $N$ Lanczos algorithms with inputs $\mathcal{A}$ and $\uno_k$, using the strategy discussed in Section  \ref{sec:lanczos:dist}. We denote this approximation by $d_n(k,\ell)$ and we compare it with the distance between $k$ and $\ell$ computed by the Matlab function \verb distances .  
 Note that  
when $k$ and $\ell$ are disconnected, $d_n(k,\ell) = n$ for every $n$, while \verb distances  correctly sets to $\infty$ the distance between $k$ and $\ell$.

In order to evaluate the proposed method, we consider the quantity 
$$\varrho_n(G) = \frac{\#\{(k,\ell) | \, d_n(k,\ell) \neq d^M(k,\ell), \, k\neq\ell, \, \dist_G(k,\ell) < \infty \}} {\#\{(k,\ell) | \, k \neq \ell, \, \dist_G(k,\ell) < \infty \} }, $$
which depends on the number of pairs $(k,\ell)$
for which $d_n(k,\ell) \neq d^M(k,\ell)$
(excluding the disconnected  and the coincident nodes). 
Table \ref{tab:dist} presents $\varrho_n(G)$ for some values of $n$
and for the networks listed above. 
The table clearly shows that for a small number of Lanczos iterations 
we are able to determine most of the distances.
Moreover, when $n$ is greater or equal to the diameter of the network and $\dist_G(k,\ell) < \infty$ we always have
$d_n(k,\ell) = d^M(k,\ell)$.

\begin{table}[t]
\begin{center}
\begin{tabular}{|c|cccc|}
\hline
   $n$   & Erd\"os   &     Facebook   &     GRQC       &     Gnutella \\
   \hline
   7  & $1.7188e-02$ &  $1.6729e-02$  &  $2.5187e-01$  &  $5.1500e-04$ \\
   9  & $5.4461e-04$ &  $7.2177e-04$  &  $3.0759e-02$  &  $5.1438e-08$ \\
   11 &     $0$      &       $0$      &  $2.3687e-03$  &      $0$       \\
   13 &     $0$      &       $0$      &  $1.0956e-04$  &      $0$       \\
   15 &     $0$      &       $0$      &  $5.4373e-06$  &      $0$       \\
   17 &     $0$      &       $0$      &       $0$      &      $0$       \\
   \hline

  \end{tabular}
  \caption{Values of $\varrho_n(G)$ for $n=7,9,11,13,15,17$
and for the networks:
Erd\"os, Facebook, GRQC, and Gnutella.}
\label{tab:dist}
\end{center}
\end{table}

Now, for each network, we select the $10$ nodes having smallest centrality $\exp(\mathcal A)_{kk}$ and we perturb the edge topology of the graph by adding all the missing edges among those nodes (obtaining a clique connecting all the ``least important'' nodes).

The plots of Figure \ref{fig:real-networks} represent the actual variation of network exp-centrality values $|\exp(\mathcal A)_{kk} - \exp(\tilde {\mathcal A})_{kk} |$ (red crosses)
and the bound in Corollary \ref{cor.exp} combined with the results in Lemma \ref{lemma:Apow:normalized} (blue circle). 
Recall that, as before, nodes in the plots are labeled according to the distance from (and to) the set $S$ of modified nodes. 
The results show that the proposed bounds, although not tight,  well approximate the actual behavior of the variation of $f(A)_{kk}$. This allows to predict the nodes whose centrality index remains effectively  unchanged under perturbations of the original graph topology. For example,  the exp-centrality of all the nodes from 3000 onwards in GRQC is guaranteed to be unchanged up to 10 digits of precision.


\begin{figure}[t]
 \centering
 \includegraphics[width=.45\textwidth, trim =  33mm 7mm 43mm 6mm, clip]{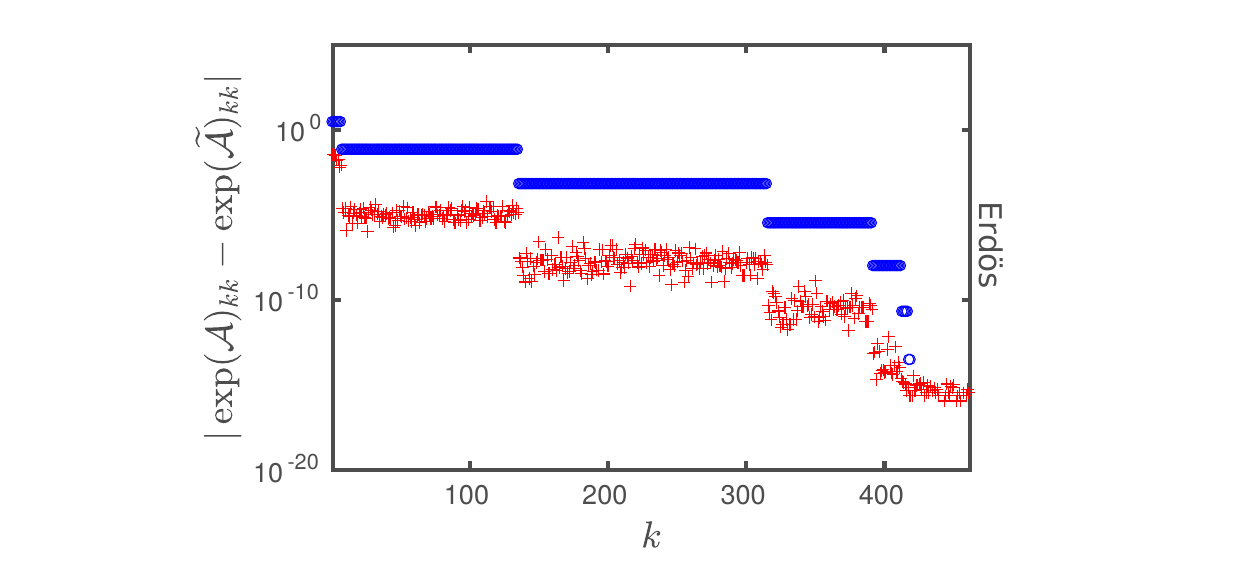}
 \includegraphics[width=.45\textwidth, trim =  22mm 7mm 28mm 5mm, clip]{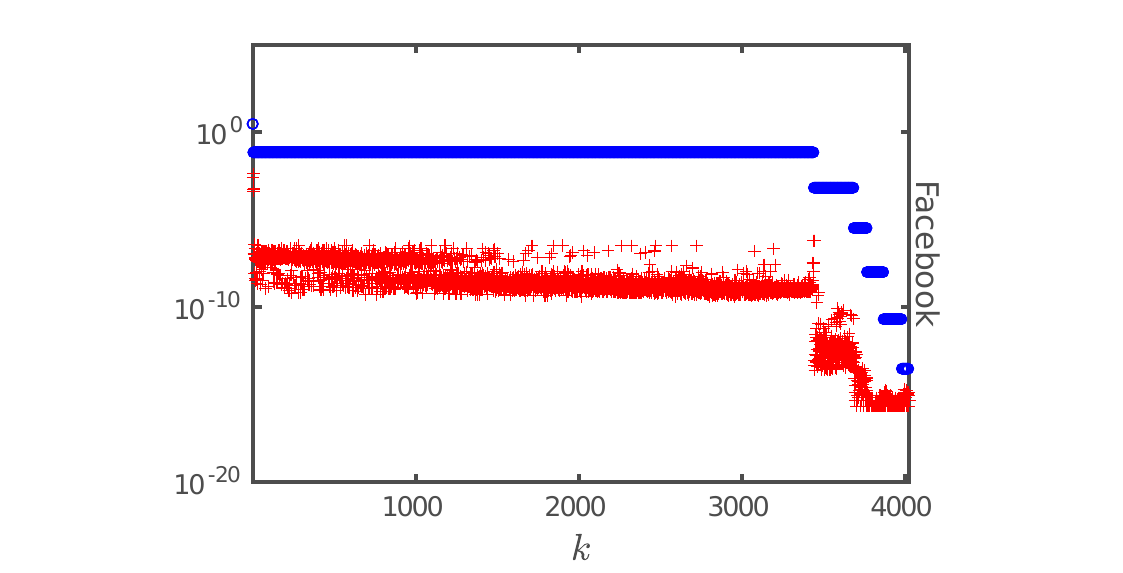}\\
 \includegraphics[width=.45\textwidth, trim =  32mm 1mm 42mm 5mm, clip]{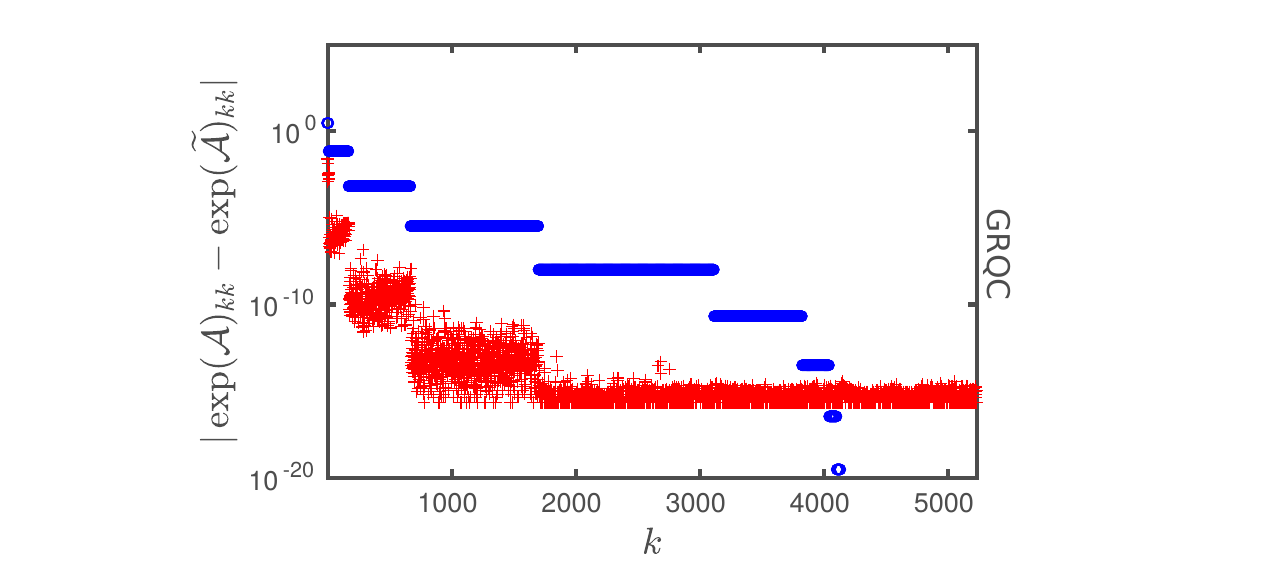}
 \includegraphics[width=.45\textwidth, trim =  22mm 1mm 27mm 6mm, clip]{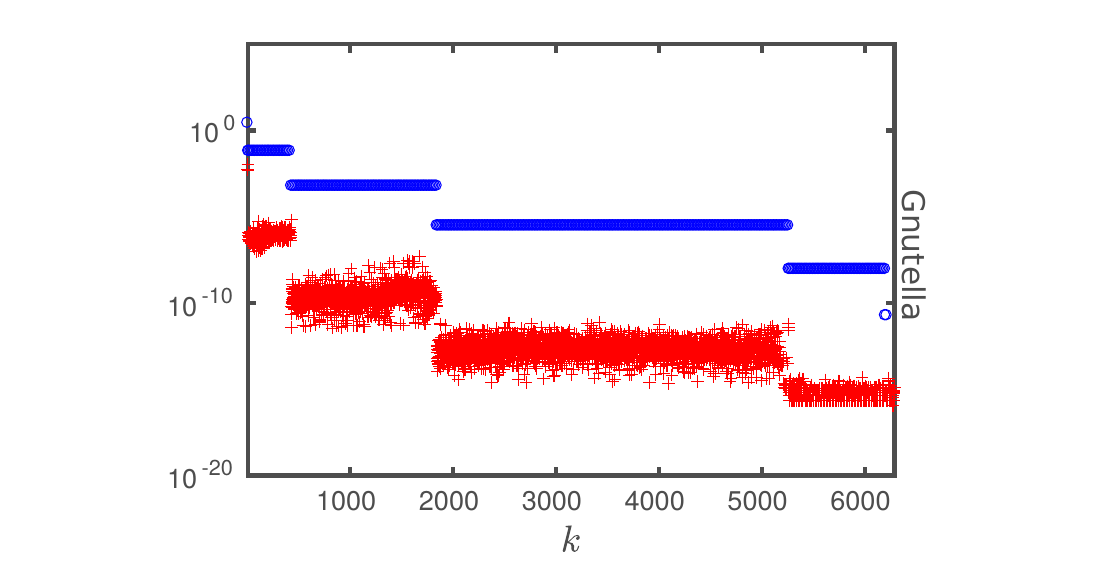}
 \caption{Absolute variation $|\exp(\mathcal A)_{kk}-\exp(\tilde{\mathcal A})_{kk}|$ when $k$ ranges from $1$ to $N$. The perturbed edges are $\delta E = S\times S$, where $S$ is the set of $10$ nodes with least $\exp$-centrality. Nodes in the plot are ordered so that the larger is $k$ the farther it is from $S$. Red crosses show the actual difference whereas the blue circles  show the bound of Corollary \ref{cor.exp} combined with Lemma \ref{lemma:Apow:normalized}.  }\label{fig:real-networks}
\end{figure}

 \section{Conclusion}
 Centrality and communicability indices based on function of matrices are among the most effective measures of the importance of nodes and of the robustness of edges in a network. These quantities are defined as the entries $f(A)_{k\ell}$ where $f(A)$ is a suitable function of a matrix $A$ describing the structure of the network $G$.   
 In this work we address the somewhat natural problem of understanding the stability of such indices with respect to perturbations in the edge topology of the graph. 
 This is important because, if $A$ is modified into $\tilde A = A+\delta A$, the entries of $f(\tilde A)$ should in principle be re-computed from scratch and this is can be a very costly operation. Thus being able to efficiently update the entries of $f(A)$ when $A$ undergoes such a perturbation is a relevant  task. When $\delta A$ has low rank, this problem can be easily addressed via the Sherman-Morrison formula for the special function $f(z)=z^{-1}$. When $f$ is a rational function this problem was considered in  \cite{bernstein2000rational}, whereas, for the case of general functions $f$, important advances in this direction have been made in \cite{beckermann2017low}, simultaneously and independently with respect to the present work.

 On a related but somewhat different line, instead, our analysis  reveals that the absolute variation of  $f(A)_{k\ell}$ decays exponentially with respect to the distance in $G$ that separates $k$ and $\ell$ from the set of nodes touched by the perturbed edges. 
 The knowledge of this behavior can be of help to practical applications. In fact,  we propose a simple numerical strategy that allows to compute the distances between nodes in $G$ simultaneously with the computation of the entries of $f(A)$, with essentially no additional cost. In particular, computing the diagonal of $f(A)$ for undirected networks (the network $f$-centrality scores) allows to compute the all-pairs shortest-path distances in the graph. Thus, using the proposed bounds, we are able to predict the magnitude of variation in the $f$-centralities of $G$ when changes occur in a localized set of edges or, vice-versa, for each node $k$ we can locate a set of nodes whose change in the edge topology affects the score $f(A)_{kk}$ by a small order of magnitude.     
 
 Examples of applications include the case where  the edge topology is evolving in time and changes in $G$ happen more frequently in sub-graphs being far  with respect to the set of nodes one is actually interested in, or  where the information on the edge structure of certain nodes is not fully reliable or, equivalently, is likely to be  affected by noise.

 Finally, the results proposed are numerically tested on some example networks borrowed from real-world applications. Our experiments show that the proposed bounds well resemble the actual behavior of the variation of $f(A)_{k\ell}$ although being some orders of magnitude larger. A clear margin for improvement and further work is thus left open, in order  to determine a better constant $c>2$ to be added in the exponent $\delta+c$ of Theorem \ref{thm.decay.tau}.  
 A possible direction in this sense is, to our opinion, the connection with Krylov methods. For instance, Theorem 5.2 in \cite{beckermann2017low} could be tailored to the graph matrix setting considered here to obtain bounds similar to the ones we have presented.

 \section*{Acknowledgment}
We are grateful to two anonymous referees for the numerous  comments and suggestions that have led to several remarkable improvements.


\end{document}